\documentclass{amsart}
\usepackage{mathrsfs}
\usepackage[all]{xy}

\usepackage{wasysym}
\usepackage{amssymb}
\usepackage{MnSymbol}
\usepackage{comment}
\usepackage{enumerate}

\usepackage{ifpdf}
\ifpdf 
  \usepackage[pdftex]{graphicx}
  \DeclareGraphicsExtensions{.pdf,.png,.jpg,.jpeg,.mps}
  \usepackage{pgf}
\else 
  \usepackage{graphicx}
  \DeclareGraphicsExtensions{.eps,.bmp}
  \usepackage{pgf}
  \usepackage{pstricks}
\fi
\usepackage{epic,bez123}
\usepackage{wrapfig}

\newtheorem{thm}{Theorem}[section]
\newtheorem{prop}[thm]{Proposition}
\newtheorem{cor}[thm]{Corollary}
\newtheorem{lem}[thm]{Lemma}

\newtheorem*{claim}{Claim}

\newtheorem{assump}{}

\newtheorem{conv}[thm]{Convention}

\theoremstyle{definition}
\newtheorem{defn}[thm]{Definition}
\newtheorem{defnconst}[thm]{Definition-Construction}

\theoremstyle{remark}
\newtheorem*{question}{Question}
\newtheorem{obs}[thm]{Observation}

\newtheorem*{rem}{Remark}
\newtheorem{example}[thm]{Example}

\newtheorem*{ack}{Acknowledgment}

\newcommand{\G }{\mathscr{G} (G, S\cup \mathcal P)}
\newcommand{\Gc }{\mathscr{\hat G} (G, S\cup \mathcal P)}

\newcommand{\Gx }{\mathscr{G} (G, S)}

\newcommand{\Gr }{\mathcal{G}}

\newcommand{\GP }{(G, \mathcal P)}

\newcommand{\act}{\curvearrowright}

\newcommand{\Gf}{\overline{G}_\lambda}
\newcommand{\pGf}{\partial_\lambda{G}}

\newcommand{\pG}{\partial{G}}
\newcommand{\cG}{\partial^c{G}}
\newcommand{\pX}{\partial X}

\newcommand{\PS}[1]{\mathcal P_{#1}(s,o)}
\newcommand{\PSv}[1]{\mathcal P_{#1}(s,v)}
\newcommand{\Ps}[1]{\mathcal P_{#1}(s,1)}

\newcommand{\PSS}[1]{\mathcal P'_{#1}(s,o)}

\newcommand{\g}[1]{\delta_{#1}}

\newcommand{\E}{\mathcal E(\epsilon, L)}
\newcommand{\GE}{G \setminus  \mathcal E(\epsilon, L)}
\newcommand{\EE}{\mathcal E(\epsilon', L)}
\newcommand{\GEE}{G \setminus \mathcal E(\epsilon', L)}

\newcommand{\Q}{\mathcal Q(\epsilon, L)}
\newcommand{\GQ}{Go \setminus \mathcal Q(\epsilon, L)}

\newcommand{\GQQ}{Go \setminus \mathcal Q(\epsilon', L)}

\newcommand{\nh}{\llangle h^n \rrangle}
\newcommand{\nkh}{\llangle h^{kn} \rrangle}

\newcommand{\diam }[1]{{\|#1\|}}

\newcommand{\proj }{\mbox{Pr}}

\newcommand{\inj }{\mbox{Inj}}

\newcommand{\ax }{\mbox{\mbox{Ax}}}

\newcommand{\dirac}[1]{{\mbox{Dirac}}{(#1)}}

\newcommand{\len }{\ell}

\usepackage[bookmarks=true, pdfauthor={YANG Wenyuan}]{hyperref}

\title[Patterson-Sullivan measures and growth of quotients]{Patterson-Sullivan measures and growth of relatively hyperbolic groups}

\author{Wen-yuan Yang}

\address{Beijing International Center for Mathematical Research \&
School of Mathematical Sciences, Peking University, Beijing, 100871,
P.R.China}
\address{Le Département de Mathématiques de la Faculté des
Sciences d'Orsay, Université Paris-Sud 11, France}

\email{yabziz@gmail.com}
\thanks{This research was supported by the ERC starting grant GA 257110 ”RaWG”.}

\begin{document}

\subjclass[2000]{Primary 20F65, 20F67}

\date{}

\dedicatory{}

\keywords{Growth, relatively hyperbolic groups, Patterson-Sullivan
measures, small cancellation}

\begin{abstract}
We prove that for a relatively hyperbolic group, there is a
sequence of relatively hyperbolic proper quotients such that their
growth rates converge to the growth rate of the group. Under natural
assumptions, similar result holds for the critical exponent of
a cusp-uniform action of the group on a hyperbolic metric space. As a
corollary, we obtain that the critical exponent of a torsion-free
geometrically finite Kleinian group can be arbitrarily approximated
by those of proper quotient groups. This resolves a question of
Dal'bo-Peign\'e-Picaud-Sambusetti.

Our approach is based on the study of Patterson-Sullivan measures on
Bowditch boundary of a relatively hyperbolic group. The uniqueness
and ergodicity of Patterson-Sullivan measures are proved when the
group is of divergent type. We also prove a variant of the Sullivan Shadow Lemma,
called \textit{Partial Shadow Lemma}. This tool allows us to prove
several results on growth functions of balls and cones. One central
result is the existence of a sequence of geodesic rooted trees with growth
rates sufficiently close to the growth rate of the group, and transition
points uniformly spaced in trees. We are then able to embed these trees into
small cancellation quotients. This proves the results in
the previous paragraph.
\end{abstract}

\maketitle

\setcounter{tocdepth}{1} \tableofcontents

\section{Introduction}
In recent years, there has been an increasing interest in the study
of relatively hyperbolic groups. The notion of relative
hyperbolicity originated in seminal work of Gromov in \cite{Gro},
and was later re-formulated and elaborated on by many authors, cf.
\cite{Farb}, \cite{Bow1}, \cite{Osin}, \cite{DruSapir}, \cite{Ge1},
to name a few.  The class of relatively hyperbolic groups includes
many naturally occurring groups, such as word hyperbolic groups,
fundamental groups of non-uniform lattices with negative curvature
\cite{Bow3}, limit groups \cite{Dah}, and CAT(0) groups with
isolated flats \cite{HruKle}.

It is a popular theme in geometric group theory to study a group via
a proper action on a model space with nice geometric properties. The
asymptotic geometry of the group action then encodes information
from the group. For example, one could simply count the number of
orbit points in a growing ball. A celebrated theorem of Gromov in
\cite{Gro2} says that the polynomial growth in a Cayley graph
implies the group being virtually nilpotent.

A group is elementary if it is finite or a finite extension of a
cyclic group. Recall that non-elementary relatively hyperbolic
groups have exponential growth, to which one can associate a real
number called \textit{growth rate} to measure the speed. The goal of
the present paper is then to investigate growth rate for quotients
of a relatively hyperbolic group which admits a proper action on
various geometric spaces.  In what follows, we describe our results
in detail.
\subsection{Growth rates of relatively hyperbolic groups}
We first prepare some general setup. Let $G$ be a group acting
properly on a proper geodesic metric space $(X, d)$. Fix a basepoint $o \in
X$. Denote $N(o, n) = \{g \in G: d(o, go) \le n\}$ for $n\in
\mathbb N$.  The \textit{growth rate} $\g A$ of a subset $A \subset
G$ with respect to $d$ is then defined as
$$
\g A = \limsup\limits_{n \to \infty} \frac{\ln \sharp (N(o, n) \cap
A)}{n}.
$$
Note that $\g A$ does not depend on the choice of $o$.

Let $\Gamma$ be a normal subgroup in $G$. The quotient
$\bar G=G/\Gamma$ is called \textit{proper} if $\Gamma$ is
infinite. Consider an action of $\bar G$ on $\bar X=X/\Gamma$ by $\Gamma g \cdot (\Gamma x) \to
\Gamma gx$. Equip $\bar X$ with the quotient metric $\bar d$. Then $\bar G$ acts properly on $\bar X$, for which we denote by $\g{\bar G}$ the growth rate of $\bar G$ with respect to $\bar d$.
\begin{defn}
The action of $G$ on $(X, d)$ is called \textit{growth tight} if $\g{G}
> \g{\bar G}$ for any proper quotient $\bar G$ of $G$.
\end{defn}
We consider here proper quotients, otherwise the equality $\g G=\g {\bar
G}$ holds.
\begin{rem}
Growth tightness was introduced by Grigorchuk and de la Harpe in
\cite{GriH} for word metric. The study of growth tightness for
general metrics was proposed by Sambusetti in \cite{Sam1}.  See
\cite{AL}, \cite{Sam2}, \cite{Sam3}, \cite{Sab}, \cite{DPPS} for
further references about growth tightness.
\end{rem}

First of all, we consider the model space arising from Cayley graphs
of $G$. Let $S$ be a finite generating set such that $1 \notin S$ and $S=S^{-1}$.
Then $G$ acts on the Cayley graph $\Gx$ endowed with the word metric
$d$. Here $\g G$ is the usual growth rate of $G$ with respect to
$S$. In \cite{YANG6}, we proved that any (non-elementary) relatively
hyperbolic group is growth tight with respect to the word metric. Hence,
a natural question arises as follows:
\begin{question}
Is there a gap between $\g G$
and $\sup\{\g{\bar G}\}$ over all proper quotients $\bar G$?
\end{question}
Our first main result is to answer this question negatively by stating the following.

\begin{thm}\label{ThmA}
Suppose that $G$ is a relatively hyperbolic group with a finite
generating set $S$. Then there exists a sequence $\bar G_n$ of
relatively hyperbolic proper quotients of $G$ such that
$$\lim\limits_{n \to \infty} \g{\bar G_n} = \g{G}.$$

\end{thm}

We compare Theorem \ref{ThmA} with some related works. In a
hyperbolic group $G$ without torsion, Coulon recently showed in
\cite{Coulon} that the growth rate of the periodic quotient $G/G^n$
converges to $\g G$ as $n$ odd approaches $\infty$. In Theorem
\ref{ThmA}, we construct $\bar G_n$ as small cancellation quotients
$G/\nh$ for certain $n\gg 0$ over a hyperbolic element $h \in G$.

In \cite{GriH}, Grigorchuk and de la Harpe asked under which
conditions the growth rate of a sequence of one-relator groups tends
to that of a free group. In \cite{Shu}, Shukhov showed that this is
true for one-relator small cancellation groups, when the length of
this relator goes to infinity.  In \cite[Corollary 2]{Erschler},
Erschler established a generalization for any small cancellation
group $G$: the growth rate of small cancellation quotient of $G$ by
adjoining a new relator $r$ goes to $\g G$ as $\len(r) \to \infty$.
Hence, Theorem \ref{ThmA} represents a further generalization in this
direction.


We next consider the model space which comes from the definition of
a relatively hyperbolic group. Relative hyperbolicity of a group
admits many equivalent formulations \cite{Farb}, \cite{Bow1},
\cite{DruSapir}, \cite{Osin}, \cite{Hru}, \cite{Ge1}, among which is
the following one stated in \cite{Hru}, generalizing the definition of a
geometrically finite Kleinian group in \cite{Marden}.

\begin{defn}\label{RHdefn}
Suppose $G$ admits a proper and isometric action on a proper
hyperbolic space $(X, d)$ such that $G$ does not fix a point in the Gromov boundary $\pX$. Denote by $\mathcal P$ the set of
maximal parabolic subgroups in $G$. Assume that there is a
$G$-invariant system of (open) horoballs $\mathbb U$ centered at
parabolic points of $G$ such that the action of $G$ on the complement
$$X \setminus \mathcal U,\; \mathcal U := \cup_{U \in \mathbb U} U$$
is co-compact. Then the pair $\GP$ is said to be \textit{relatively
hyperbolic}, and the action of $G$ on $X$ is called
\textit{cusp-uniform}.
\end{defn}

In this setting, the growth rate $\g G$ for the action of $G$ on $X$ is
usually called the \textit{critical exponent}, which has been studied
for a long time in the fields of Kleinian groups, dynamical systems, and
so on. It is known that parabolic subgroups are quasi-isometrically
embedded with respect to word metric, cf \cite{Osin}.  By contrast, the
asymptotic behavior of the action of $G$ on $X$ even modulo a compact part can be
still quite complicated due to the exponential distortion of parabolic
subgroups in $X$.  In \cite{DOP}, Dal'bo-Otal-Peign\'e introduced a
parabolic gap property to resolve this issue. By definition, $G$ has a \textit{parabolic gap property} (PGP) if $\g
G > \g P$ for every maximal parabolic subgroup $P \in \mathcal P$. For a quotient $G\to \bar G$ we could also say that $\bar G$ has a \textit{PGP property} if $\g {\bar P}\le \g {\bar G}$ for the image  $\bar P$ of every $P\in \mathcal P$.

In \cite{DPPS}, Dal'bo-Peign\'e-Picaud-Sambusetti considered growth
tightness of a geometrically finite discrete group of isometries of a Cartan-Hadamard manifold. They generalized the co-compact case in
\cite{Sam3} and proved that $G$ is growth tight, provided that $G$
satisfies the parabolic gap property. Thus, we can also consider an analogue question
for growth gap with respect to the hyperbolic metric. Parallel to Theorem
\ref{ThmA}, the following theorem asserts the same phenomenon for growth
rates with respect to hyperbolic metric.

\begin{thm}\label{ThmB}
Suppose that the action of $G$ on $X$ is cusp-uniform with $\g G <
\infty$. Assume that $G$ satisfies the parabolic gap property. Then
there exists a sequence $\bar G_n$ of relatively hyperbolic proper
quotients of $G$ such that $$\lim_{n\to \infty} \g {\bar G_n} = \g
{G},
$$
and $\bar G_n$ has the parabolic gap property.

\end{thm}
\begin{rem}
It is possible that $\g G=\infty$, see Example 1 in \cite[Section
3.4]{GabPau}. We do not need the growth tightness of $G$ with
respect to $d$, although it appears to the author that the proof for
growth tightness in \cite{DPPS} is valid under
the general assumption of Theorem \ref{ThmB}.
\end{rem}
\begin{rem}
Recall that the growth rate is not a quasi-isometric invariant. Hence,
Theorem \ref{ThmB} cannot be deduced directly, even in the case of hyperbolic groups,
from Theorem \ref{ThmA}. On the other hand, the overall strategy in proving Theorem
\ref{ThmB} does share much similarity with that in Theorem \ref{ThmA}.
\end{rem}

The assumption of the ``parabolic gap property'' is sharp in the following sense:
\begin{prop}\label{ThmBSharp}
Suppose that the action of $G$ on $X$ is cusp-uniform with $\g G <
\infty$. If $G$ is of divergent type without parabolic gap property, then there  exist infinitely many  distinct relatively hyperbolic proper quotients $\bar G_n$ of $G$, without parabolic gap property,  such that $$\g {\bar G_n} = \g {G}.$$
\end{prop}
\begin{rem}
A divergent Schottky group $G$ without parabolic gap property was constructed by Peign\'e in \cite[Main Theorem]{Peigne}. By making use of \cite[Lemma 5.1]{DPPS} and technics therein one could prove Proposition \ref{ThmBSharp} for Peign\'e's examples. The novelty here is that the conclusion holds for every (not necessarily free) $G$ without parabolic gap property.
\end{rem}

It is well-known that a geometrically finite Kleinian group has the parabolic gap property, cf. \cite{DPPS}. Hence, we obtain the
following corollary, which settles a question of
Dal'bo-Peign\'e-Picaud-Sambusetti in \cite{DPPS} asking whether
there is a gap between $\g G$ and $\sup\{\g {\bar G}\}$ over all
proper quotients $\bar G$.
\begin{cor}
Let $G$ be a torsion-free geometrically finite Kleinian group.  Then $$\inf\{\g G -\g{\bar G}\} = 0$$ over all proper
quotients $\bar G$ of $G$.
\end{cor}

\subsection{Patterson-Sullivan measures} We now describe
the approach in proving Theorems \ref{ThmA} and \ref{ThmB}. In
\cite{Coulon} and \cite{Erschler}, the typical approach consists of
two components: the automata theory of a hyperbolic group (or
equivalently, Cannon's theory about finiteness of cone types in
\cite{Cannon}), and small cancellation theory. Unfortunately, there
is no automata that can recognize the geodesic language of any
relatively hyperbolic group.

The tool that we find to replace the automata theory turns out to be
the theory of Patterson-Sullivan measures (or PS-measures for
shorthand). In fact, the idea of using PS-measures to study growth
problems dates back to works of Patterson \cite{Patt} and Sullivan
\cite{Sul}, see \cite[Theorem 9]{Sul} for example.

In \cite{Coor}, Coornaert has established the theory of PS-measures
on the limit set of a discrete group acting on a $\delta$-hyperbolic
space.  In a discrete group of isometries of a Cartan-Hadamard manifold, Dal'bo-Otal-Peign\'e used the parabolic gap property in
\cite{DOP} to deduce the divergence of $G$, and demonstrated that
this property can be very helpful to obtain a satisfactory theory of
PS-measures.  A significant part of our study is to further develop
their idea in a general setting.

Let $G$ be a relatively hyperbolic group with a finite generating
set $S$. The
\textit{Bowditch boundary}, denoted by $\pG$,  is the
(topological) Gromov boundary $\pX$ of $X$ in Definition
\ref{RHdefn}, since Bowditch \cite{Bow1} proved that $\pG$ is
independent of the choice of $X$. In \cite{Ge1}, Gerasimov proved
that $\pG$ can also be obtained as quotients of Floyd boundary. This gives a way to
compactify the Cayley graph $\Gx$ by $\pG$.  See Section \ref{Section2} for more discussion.

The main novelty of the present study is to construct PS-measures $\{\mu_v\}_{v\in G}$ on $\pG$ through the action of $G$ on $\Gx$ for studying the growth rate. It is well-known that the construction of PS-measures can be
performed for any proper group action on a geodesic metric space on the horofunction boundary, cf. \cite{BuMo}. The power of the PS-measure theory depends somehow how a geometric boundary under consideration is ``close'' to the horofunction boundary. For instance, the Gromov boundary of  a hyperbolic space differs up to a finite part to the horofunction boundary so that PS-measures is a quasi-conformal density. In particular, recall that it will become as a conformal density when the space is CAT(-1): this is the classical case considered by Patterson, Sullivan and many other authors.

In contrary to the hyperbolic case, our difficulty in the setting of Cayley graph to prove the quasi-conformal density of PS-measures on $\pG$ is the lack of a
well-defined notion of horofunctions at all boundary points in
$\pG$. We get around this issue by defining horofunctions only
at conical points and, at the same time, prove that PS-measures have no atoms.  We summarize the results about PS-measures as
follows. See Section \ref{Section4} for detail and precise definitions.
\begin{thm}[=Proposition \ref{PSMeasureG}]\label{ThmP} Let $\{\mu_v\}_{v\in G}$ be a PS-measure on $\pG$ constructed
through the action of $G$ on $\Gx$. Then $\{\mu_v\}_{v\in G}$ is a quasi-conformal
density without atoms. Moreover, they are unique and ergodic.
\end{thm}
\begin{rem}
In \cite{YANG8} we carried out a detailed study of PS-measures constructed using a cusp-uniform action on $X$. Most results in this paper have an analogue there.  
\end{rem}

In the theory of PS-measures, a key tool is the Sullivan Shadow
Lemma, which connects the geometry inside and the measure on
boundary. In our setup of group completion $\Gx \cup \pG$, the
Sullivan Shadow Lemma is proved. In fact, we prove a variant of
the Shadow Lemma that holds for \textit{partial shadows}. Before stating
our result, we have to go into some technical parts of this paper.

Define $\mathbb P=\{gP: g \in G, P \in \mathcal{\bar P}\}$, where
$\mathcal{\bar P}$ is a prefered complete set of conjugacy representatives in
$\mathcal P$. For a path $p$ in $\Gx$, a point $v \in p$ is called
an \textit{$(\epsilon, R)$-transition} point for $\epsilon, R \ge 0$
if the $R$-neighborhood $p \cap B(v, R)$ around $v$ in $p$ is not
contained in the $\epsilon$-neighborhood of any $gP \in \mathbb P$.
The notion of transition points was due to Hruska in \cite{Hru},
and further elaborated on by Gerasimov-Potyagailo in \cite{GePo4}.

Hence, the \textit{partial shadow} $\Pi_{r,\epsilon, R}(g)$ at $g$
for $r \ge 0$ is the set of boundary points $\xi \in \pG$ such that
some geodesic $[1, \xi]$ intersects $B(g, r)$ and
contains an $(\epsilon, R)$-transition point $v$ in $B(g, 2R)$. We
prove the following variant of the Shadow Lemma. 
\begin{lem}[=Partial Shadow Lemma \ref{PShadowLem}]\label{PSLem}
There are constants $r_0,
\epsilon, R\ge 0$ such that the following holds
$$
\exp(-\g G d(1, g)) \prec \mu_1(\Pi_{r, \epsilon, R}(g)) \prec_r
\exp(-\g G d(1, g)) ,
$$
for any $g \in G$ and $r \ge r_0$.
\end{lem}

By Partial Shadow Lemma, we are able to prove a series of results about the
growth functions of balls and cones as we describe now.

The \textit{cone} $\Omega_r(g)$ at $g$ for $r\ge 0$ in $\Gx$ is the
set of elements $h \in G$ such that some geodesic $[1, h]$
intersects $B(g, r)$. The notion of a partial cone
$\Omega_{r, \epsilon, R}(g)$ at $g$ is defined similarly, by
demanding the existence of $(\epsilon, R)$-transition points on $[1,
h]$ $2R$-close to $g$. 

\begin{thm}[=Propositions \ref{ballgrowth}, \ref{PConeGrowthG}]\label{ThmG} Let $r,\epsilon, R\ge 0$ be given by
Lemma \ref{PSLem}.  Then the following holds.

\begin{enumerate}
\item
$\sharp N(1, n) \asymp \exp(n\g G)$ for any $n \ge 0$.
\item
$\sharp (\Omega_{r, \epsilon, R}(g)\cap N(g, n)) \asymp \exp(n\g
G)$ for any $g \in G$.
\item
In particular, $\sharp (\Omega_{r}(g) \cap N(g, n)) \asymp \exp(n\g
G)$ for any $g \in G$.
\end{enumerate}
\end{thm}
\begin{rem}
The statements (1) and (3) in hyperbolic case were proved in
\cite[Th\'eor\`eme 7.2]{Coor} and \cite[Lemma 4]{AL} respectively. The formula (1) is also an ingredient in  proving statistical hyperbolicity of a relatively hyperbolic group \cite{OYANG}. 

For cusp-uniform actions, it is shown in \cite{YANG8} that the purely exponential growth type as  (1), (2) and (3) is equivalent to a condition introduced by Dal'bo-Otal-Peign\'e, which is implied by the parabolic gap property. See Section \ref{Section7} for details.

\end{rem}

As a consequence of Theorem \ref{ThmG}.(3), the following extends a result of Bogopolski \cite[Lemma 3]{Bog} about dead-ends  to the relative hyperbolic case. See a definition for the dead-end depth in Subsection \ref{SSpartialconegrowth}. 
\begin{cor}[Dead-end depth is uniform]
The dead-end depth of  any $g\in G$ is at most $r$ where $r>0$ is given by Theorem \ref{ThmG}.
\end{cor}

Another result of independent interest is that there are finitely many partial
cone types $\Omega_{0, \epsilon, R}(g)$ for fixed $\epsilon, R$. See Lemma \ref{FinPCones}. This could be seen as an analogue of a theorem of Cannon \cite{Cannon} in hyperbolic groups. In \cite{PYANG}, we make a crucial use of this result  to construct a ``symmetric'' version of Theorem \ref{ThmC} below. \\

\paragraph{\textbf{Outline of proofs of Theorems \ref{ThmA} \&
\ref{ThmB}}}Let's now describe the strategy of the proof of Theorem \ref{ThmA}. The proof of
Theorem \ref{ThmB} is similar.  The central result in proving Theorem \ref{ThmA} is the
existence of a sequence of sufficiently large trees with uniformly
spaced transitions points. 

We say that a \textit{rooted geodesic tree} in a graph is a tree $\mathcal T$ with a distinguished vertex $o$ such that every path in $\mathcal T$ starting at $o$ is a geodesic in the graph.

\begin{thm}[=Theorem \ref{LargeTreeG}]\label{ThmC}
Suppose that $\GP$ is a relatively hyperbolic group with a finite
generating set $S$. Then there exist $\epsilon, R>0$ such that the
following holds.

For any $0 <\sigma < \g G$, there exist $r >0$ and a rooted geodesic tree
$\mathcal T$ at $1$ in $\Gx$ with the following properties.
\begin{enumerate}
\item
Let $\gamma$ be a geodesic in $\mathcal T$ with $\gamma_-=1$. For any $x \in \gamma$,
there exists an $(\epsilon, R)$-transition point $v$ in $\gamma$
such that $d(x, v)<r$.
\item
Let $\g{\mathcal T}$ be the growth rate of $\mathcal T$. Then
$\g{\mathcal T} \ge \sigma$.
\end{enumerate}
\end{thm}

\begin{rem}
We remark that if
$\sigma$ is closer to $\g G$, then $r$ is bigger. Consequently,
transition points get sparser, and thus geodesics in $\mathcal T$ get
wilder and less controlled.
\end{rem}

The other ingredient in proving Theorem \ref{ThmA} is small
cancellation theory. The idea then is to embed sufficiently large trees
given by Theorem \ref{ThmC} into small cancellation quotients
constructed as follows.

Fix a hyperbolic element $h \in G$, and denote by $E(h)$
the maximal elementary subgroup in $G$ containing $h$. We consider the quotients $G \to \bar G:= G/\nh$ for certain sufficiently large $n$. The general philosophy behind small cancellation theory is that any element in $\bar G$ cannot contain ``major'' parts of elements in $E(h)$. Such type of results are achieved here by exploiting the theory of rotating family recently developed by Dahmani-Guirardel-Osin in \cite{DGO}.  In Proposition \ref{injective}, we identify a subset of elements in $G$ embeddable into $\bar G$ such that these elements cannot contain sufficiently deep elements in $E(h)$. 

On the other hand, since $G$ is hyperbolic relative to $\mathcal P \cup \{gE(h): g\in G\}$, we could apply Theorem \ref{ThmC} to this new peripheral structure and obtain a sequence of geodesic trees $\mathcal T$ such that $\g {\mathcal T} \to \g G$ and $\mathcal T$ contain ``many" transition points, not deep in $E(h)$. By Proposition \ref{injective}, these transition points force $\mathcal T$ to be injecting into $\bar G$ for $n\gg 0$.  So $ \g{\bar G} \ge \g {\mathcal T}$ and  $\g{\bar G} \to \g G$. This proves  Theorem \ref{ThmA}. 
Their detailed proofs can be found in Section \ref{Section9}.

\subsection{Organization of paper}
In Section \ref{Section2},  we discuss a notion of transition points with respect to a contracting system. Several preliminary results are proved for conical points and horofunctions. Section \ref{Section3} then gives an abstract formulation of the quasi-conformal density and describes the construction of PS-measures. The key notions are partial shadows and cones in this paper.

Sections  \ref{Section4} \& \ref{Section5}  focus on the PS-measures on the compactification $\Gx
\cup \pG$ and use them to derive results about growth of balls and
cones in Cayley graphs. Theorems \ref{ThmP}, \ref{ThmG} and
\ref{ThmC} are proved.

The aim of Section \ref{Section6} is to link the action of $G$ on $\Gx$ to the
action of $G$ on $X$, via \textit{lift} operations. Section \ref{Section7} is devoted  to  an analogue of
Theorem \ref{ThmC} for the
cusp-uniform action of $G$ on $X$.

Section \ref{Section8} studies the small cancellation over hyperbolic elements of
high power. With all ingredients prepared,  Theorems \ref{ThmA} and \ref{ThmB} are proved in Section \ref{Section9}.

We remark that the results and technics in this paper have found applications in the forth-coming papers \cite{OYANG}, \cite{PYANG}. 

\ack The author would like to thank Anna Erschler for bringing to
his attention the reference \cite{Erschler}, which initiated the
study of this paper.  Thanks also go to Leonid Potyagailo for
suggesting the sufficient part of Lemma \ref{charconical} and pointing out the example in
\cite{GabPau}, and Fran\c{c}ois Dahmani for helpful conversations.

\section{Preliminaries}\label{Section2}

\subsection{Notations and Conventions}

Let $(X, d)$ be a geodesic metric space. We collect some notations
and conventions used globally in the paper.
\begin{enumerate}
\item
$B_d(x, r) := \{y: d(x, y) \le r\}$. The index $d$, if understood, will be omitted.
\item
$N_r(A): = \{y \in X: d(y, A) \le r \}$ for a subset $A$ in $X$.
\item
$\diam {A}$ denotes the diameter of $A$ with respect to $d$.
\item
We always consider a rectifiable path $p$ in $X$ with arc-length parametrization.  Denote by $\len (p)$ the length
of $p$.  Then $p$ goes from  the initial endpoint $p_-$ to the terminal endpoint $p_+$.  Let $x, y \in p$ be two points which are given by parametrization. Then denote by $[x,y]_p$ the parametrized
subpath of $p$ going from $x$ to $y$.
\item
Given a property (P), a point $z$ on $p$ is
called the \textit{first point} satisfying (P) if $z$ is among the
points $w$ on $p$ with the property (P) such that $\len([p_-, w]_p)$
is minimal. The \textit{last point} satisfying (P) is defined analogously.

\item
For $x, y \in X$, denote by $[x, y]$ a geodesic $p$ in $X$ with
$p_-=x, p_+=y$. Note that the geodesic between two points is usually
not unique. But the ambiguity of $[x, y]$ is usually made clear or
does not matter in the context.

\item
Let $p$ a path and $Y$ be a closed subset in $X$ such that $p \cap Y
\neq \emptyset$. So the \textit{entry} and \textit{exit} points of
$p$ in $Y$ are defined to be the first and last points $z$ in $p$
respectively such that $z$ lies in $Y$.


\item
Let $f, g$ be two real-valued functions with domain understood in
the context. Then $f \prec_{c_1, c_2, \cdots, c_n} g$ means that
there is a constant $C >0$ depending on parameters $c_i$ such that
$f < Cg$.  And $f \succ_{c_1, c_2, \cdots, c_n} g,  f \asymp_{c_1,
c_2, \cdots, c_n} g$ are used in a similar way. For simplicity, we
omit $c_i$ if they are uniform.
\end{enumerate}

Recall that in a geodesic triangle, two points $x, y$ in sides $p, q$ respectively are
called \textit{congruent} if $d(x, o)=d(y, o)$ where $o$ is the
common endpoint of $p$ and $q$. Define the Gromov product $(x, y)_z=(d(x, z)+d(y, z)-d(x, y))/2$ for $x, y, z\in X$.

We make use of the following definition of hyperbolic spaces in the sense of Gromov.
\begin{defn}\label{thintri}
A geodesic space $(X, d)$ is called \textit{$\delta$-hyperbolic} for $\delta \ge 0$ if any geodesic triangle is \textit{$\delta$-thin}: let $p, q$ be any two sides such that $o:=p_-=q_-$. Then a point $x$ in $p$ such that $d(x, p_-)\le (p_+,q_+)_o$ is $\delta$-close to a congruent point in $q$.
\end{defn}

\subsection{Contracting systems}Some material is borrowed from \cite{YANG6}.
The terminology of a contracting set is due to Bestvina-Fujiwara in
\cite{BF2}. We also refer to \cite{Sisto} for an alternative
formulation of contracting sets.

Given a subset $Y$ in $X$,  the projection $\proj_Y(x)$ of a point
$x$ to $Y$ is the set of the nearest points in the closure of $Y$ to
$x$. Then for $A \subset X$ define $\proj_Y(A) = \cup_{a\in A
}\proj_Y(a)$.

\begin{defn}\label{contractdefn}
Let $\tau, D>0$. A subset $Y$ is called \textit{$(\tau,
D)$-contracting} in $X$ if the following holds
$$\diam{\proj_Y(\gamma)} < D$$
for any geodesic $\gamma$ in $X$ with $N_\tau(Y) \cap \gamma =
\emptyset$.

A collection of $(\tau, D)$-contracting subsets is referred to as a
$(\tau, D)$-\textit{contracting system}. The constants $\tau, D$
will often be  omitted, if understood. \end{defn}

We now recall some useful properties of contracting sets. The proof of the following lemma is straightforward by considering projections and left to the interested reader. 

\begin{lem}\label{contracting}
Let $Y$ be a contracting set in $X$. Then the following holds.
\begin{enumerate}
\item
For any $r \ge 0$, there exists
$\epsilon=\epsilon(r)>0$ with the following property.  If $\gamma$ is a geodesic in $X$ such that
$d(\gamma_-, Y), d(\gamma_+, Y) < r$, then $\gamma \subset N_\epsilon(Y)$.
\item
There exists a constant $M>0$ such that for any geodesic $\gamma$
with $\gamma_- \in Y$, we have $d(\proj_Y(\gamma_+),  \gamma) \le
M$.

\item
There exists a constant $M>0$ such that $\diam{\proj_Y(\gamma)} \le
\len(\gamma) +M$ for any geodesic $\gamma$ in $X$.
\end{enumerate}
\end{lem}

\begin{lem}\label{bipbpp}\cite[Lemma 2.7]{YANG6}
Let $\mathbb Y$ be a contracting system. Then the following two
properties are equivalent.
\begin{enumerate}
\item
\mbox{(bounded intersection property)} For any $\epsilon >0$
there exists $R =R(\epsilon)>0$ such that
$$
\diam{N_\epsilon(Y) \cap N_\epsilon(Y')} < R$$ for any two distinct
$Y, Y' \in \mathbb Y$.
\item\mbox{(bounded projection property)} There exists a finite number
$D>0$ such that
$$\diam{\proj_Y(Y')} < D$$
for any two distinct $Y, Y' \in \mathbb Y$.
\end{enumerate}
\end{lem}

We describe a typical setting which gives rise to a contracting
system with the bounded intersection (and projection) property. All
examples in this paper are of this sort. The proof of the following lemma is
straightforward.
\begin{lem}\label{contractingsystem}
Suppose that a group $G$ acts properly and isometrically on a
geodesic metric space $X$. Let $\mathbb Y$ be a $G$-invariant
contracting system such that $\sharp (\mathbb Y/G) < \infty$ and for
each $Y \in \mathbb Y$ the stabilizer $G_Y$ of $Y$ in $G$ acts
co-compactly on $Y$. Assume that $\diam{N_r(Y) \cap N_r(Y')} <\infty$
for any $r>0$ and all $Y \neq Y' \in \mathbb Y$.  Then $\mathbb Y$
has the bounded intersection property.
\end{lem}

\begin{rem}
Suppose that $X$ is a hyperbolic space or the Cayley graph of a
relatively hyperbolic group $G$. The condition``$\diam{N_r(Y) \cap
N_r(Y')} <\infty$" is satisfied if the limit set of $Y$ is disjoint
with that of $Y'$.
\end{rem}

In the following definition, we state an abstract formulation of a
notion of transition points, which was introduced by
Hruska \cite{Hru}  originally in the setting of relatively hyperbolic groups.
\begin{defn}
Let $\mathbb Y$ be a contracting system with the bounded intersection property in
$X$. Fix $\epsilon, R>0$.  Given a path $\gamma$, we say that a point $v$ in $\gamma$ is called \textit{$(\epsilon,
R)$-deep} in $Y \in \mathbb Y$ if it holds that $\gamma \cap B(v, R) \subset
N_\epsilon(Y).$ If $v$ is not $(\epsilon, R)$-deep in any $Y \in
\mathbb Y$, then $v$ is called an \textit{$(\epsilon, R)$-transition
point} in $\gamma$.
\end{defn}
\begin{rem}
The above definition differs a little bit with the one in \cite{Hru}, for which $v$ is required to be $R$-apart from the endpoints of $\gamma$. One of reasons was possibly that $v$ will then be $(\epsilon, R)$-deep in a \textbf{unique} $Y\in \mathbb Y$ for $R>\mathcal R(\epsilon)$. Here, the same uniqueness property holds for $v$ whenever the length of $\gamma$ is bigger than $\mathcal R(\epsilon)$. 
\end{rem}

\begin{example}\label{contractingsystem} In the paper, we shall consider transition points in
the following setups.
\begin{enumerate}
\item
Let $\GP$ be a relatively hyperbolic group with a finite generating
set $S$. Then the collection $\mathbb P$ of peripheral cosets is
a contracting system with the bounded intersection property in $\Gx$, cf.
\cite{DruSapir}, \cite{GePo4}.

\item
Let $G$ acts properly on a hyperbolic space $X$. Assume that there
exists a collection of horoballs $\mathbb U$ in $X$ such that the action of $G$
on $X\setminus \cup_{U \in \mathbb U} U$ is co-compact. Then
$\mathbb U$ is a contracting system with the bounded intersection property.

\item
We could also consider an ``extended" relative hyperbolic structure of $G$,
which is obtained by adjoining into $\mathcal P$ a collection of
subgroups $\mathcal E$. This can be done in the following way. Let $h
\in G$ be a hyperbolic element. Denote by $E(h)$ the stabilizer in
$G$ of the fixed points of $h$ in $\pX$. Then $\mathcal E = \{g
E(h)g^{-1}: g \in G\}$ gives such an example. See \cite{Osin2} for more
detail.

Let $C(E)$ denotes the convex hull in $X$ of $\Lambda(E)$ for each
$E \in \mathcal E$. Denote $\mathbb Q = \{C(E): E \in \mathcal E\}$ and $\mathbb E=\{g
E(h): g \in G\}$.
Then $\mathbb U \cup \mathbb Q$ is a contracting system with bounded
intersection in $X$ and $\mathbb P \cup \mathbb E$ is so in $\Gx$. 

\end{enumerate}
\end{example}

Even in this abstract setting, we can still obtain the following
non-trivial facts.

\begin{lem}\label{uniformtrans}
There exists $\epsilon_0 >0$ such that for any $R>0$, there exists $L=L(R)
>0$ with the following property.

Let $\gamma$ be a geodesic in $X$, and $z \in \gamma$ such that
$d(z, \gamma_-) - d(\gamma_-, Y)>L,\; d(z, \gamma_+) -d(\gamma_+, Y)>L$ for some $Y\in \mathbb Y$.  Then $z$ is $(\epsilon_0, R)$-deep in $Y$.
\end{lem}

\begin{proof}
Assume that  $\mathbb Y$ is a $(\tau, D)$-contracting system for
$\tau, D >0$.  Set $L = \tau + D + R$.

Observe that $\gamma \cap N_\tau(Y) \neq \emptyset$. If not, we get $\diam{\proj_Y(\gamma)} < D$.  By assumption, we have that $d(\gamma_-, \gamma_+) \ge 2L+ d(\gamma_-, Y)+d(\gamma_+, Y)$ and, by the triangle inequality, $d(\gamma_-, \gamma_+) \le d(\gamma_-, Y)+d(\gamma_+, Y) + \diam{\proj_Y(\gamma)}$. Thus, $\diam{\proj_Y(\gamma)}  \ge 2L$. This is a contradiction with the choice of $L$. 
\begin{figure}[htb]
\centering \scalebox{0.5}{
\includegraphics{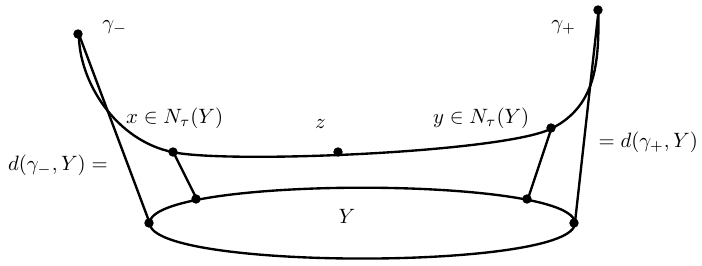} 
} \caption{Lemma \ref{uniformtrans}} \label{Figure0}
\end{figure}

Denote by $x, y$
the entry and exit points of $\gamma$ in $N_\tau(Y)$ respectively. Let $\epsilon_0 =\epsilon(\tau)$ be given by Lemma
\ref{contracting} such that $[x, y]_\gamma \subset N_{\epsilon_0}(Y)$.
By projection we have that $d(x, \proj_Y(\gamma_-)) \le \proj_Y([\gamma_-,x]_\gamma) + \tau\le D+\tau$ and so $|d(\gamma_-, \proj_Y(\gamma_-)) - d(\gamma_-, x)| \le D +
\tau$. Since $d(z, \gamma_-) - d(\gamma_-, Y)>L>D+\tau$, we thus have $z\in [x, \gamma_+]_\gamma$: $d(x, \gamma_-)\le d(z, \gamma_-)$. So $d(z, x)=d(z, \gamma_-)-d(x, \gamma_-)\ge L-D-\tau \ge R$. The same reasoning gives that $d(z, y)\ge R$.  Consequently, we have that $z$ is $(\epsilon_0,
R)$-deep in $Y$.
\end{proof}

\begin{lem}\label{entrytrans}
For any
$\epsilon > 0$, there exists $R=R(\epsilon)>0$ with the
following property.

For a geodesic $\gamma$ and $Y \in \mathbb Y$,  the corresponding entry and exit points $x, y$ of $\gamma$ in $N_\epsilon(Y)$ are $(\epsilon, R)$-transition points in
$\gamma$, provided that $d(x,y)\ge R$.
\end{lem}
\begin{proof}
By Lemma \ref{contracting}, there exists $\sigma \ge\epsilon$ such
that $[x, y]_\gamma \subset N_{\sigma}(Y)$. Set $R=R(\sigma)+1$,
where $R(\sigma)$ is given by Lemma \ref{bipbpp} for $\mathbb Y$.

Since $x$ is not $(\epsilon,
R)$-deep in $Y$, lets assume by contradiction that $x$ is $(\epsilon, R)$-deep in  $Y'
 \in \mathbb Y$ for $Y'\neq Y$. Since $d(x, y) \ge R$, we see that $N_\sigma(Y)$ and $N_\sigma(Y')$ contain a common segment of $[x,y]_\gamma$ with length at least $R$. This gives a contraction, as $R >R(\sigma)$ and $Y
\neq Y'$. Hence, $x$ is $(\epsilon,
R)$-transitional.
\end{proof}
 
\subsection{Relatively hyperbolic groups}
We consider a non-elementary relatively hyperbolic group
$\GP$ given by Definition \ref{RHdefn}. We take \cite{Gro},  \cite{GH} for general references on hyperbolic spaces, and \cite{Bow2}, \cite{Bow1}  on relatively hyperbolic groups.

Since $G$ acts isometrically on $X$, all elements in $G$ can be classified into three mutually exclusive classes: \textit{elliptic}, \textit{hyperbolic} and
\textit{parabolic} elements. We refer the reader to
\cite{Gro} for precise definitions. The basic fact that we need for a hyperbolic element $h$ is that $h$ possesses an $\langle h \rangle$-invariant bi-infinite quasi-geodesic, called the \textit{axe} of $h$, such that its two endpoints in Gromov boundary of $X$ are fixed by $h$.

\begin{defn}
Let $T$ be a compact metrizable space on which a group $G$ acts by
homeomorphisms. We recall the following definitions, cf. \cite{Bow2}.
\begin{enumerate}
\item
The action of $G$ on $T$ is a \textit{convergence group action} if the induced
group action of $G$ on the space of distinct triples over $T$ is
proper.
\item
The \textit{limit set} $\Lambda(\Gamma)$ of a subgroup $\Gamma
\subset G$ is the set of accumulation points of all $\Gamma$-orbits
in $T$.
\item
A point $\xi \in T$ is called \textit{conical} if there are a
sequence of elements $g_n \in G$ and a pair of distinct points $a, b
\in T$ such that the following holds
$$g_n(\xi, \zeta) \to (a, b),$$
for any $\zeta \in T \setminus \xi$.
\item
A point $\xi \in T$ is called \textit{bounded parabolic} if the
stabilizer $G_\xi$ in $G$ of $\xi$ is infinite, and acts properly
and co-compactly on $T\setminus \xi$.
\item
A convergence group action of $G$ on $T$ is called \textit{geometrically
finite} if every point $\xi \in T$ is either a conical point or a
bounded parabolic point.
\end{enumerate}
\end{defn}

Let $\pX$ be the Gromov boundary of $X$. In \cite{Bow1}, Bowditch showed that $G$ acts geometrically finitely
on $\pX$, which
are topologically the same for every choice of $X$ in Definition
\ref{RHdefn}. So we can associate to the pair $\GP$ a topological boundary $\pX$ called \textit{Bowditch
boundary}, denoted by $\pG$.

It is a useful fact that the stabilizer $G_U$ of $U \in \mathbb U$ acts co-compactly on the (topological) boundary $\partial U$ in $X$. Note that the boundary at infinity of a horoball $U$ in $X \cup \pX$
consists of a bounded
parabolic point fixed by $G_U$. 

This gives the following observation which will be
invoked implicitly.

\begin{obs}
Let $\xi$ be a conical point in $\pX$. Then for any $\epsilon >0$,
any geodesic ending at $\xi$ exits the $\epsilon$-neighborhood of
any horoball $U \in \mathbb U$ which the geodesic enters into.
\end{obs}

\subsection{Floyd boundary}
Recall that a relatively hyperbolic group admits a Bowditch boundary which depends on its peripheral structure. In this subsection, we
shall describe a construction of ``absolute"  boundary, due to Floyd, for a
finitely generated group.

Let $G$ be a group with a finite generating set $S$. Assume that
$1\notin S$ and $S=S^{-1}$.  Consider the Cayley graph $\Gx$  of $G$ with respect to $S$. Denote by $d_S$ (or simply by $d$ if no ambiguity)
the combinatorial metric on $\Gx$.


Fix $f(n) = \lambda^{-n}$ ($\lambda >1$). For any $v \in G$,
we define a Floyd metric $\rho_v$ as follows. The \textit{Floyd
length} $\len_v(e)$ of an edge $e$ in $\Gx $ is set to be $f(n)$, where $n =
d(v, e)$. This naturally induces a Floyd length $\len_v$ of a path and then a \textit{Floyd metric} $\rho_v$ on $\Gx$.
Let $\Gf$ be the Cauchy completion of $G$ with respect to $\rho_v$.
The complement $\pGf$ of $G$ in $\Gf$ is called \textit{Floyd
boundary} of $G$. The $\pGf$ is called \textit{non-trivial} if
$\sharp \pGf>2$. We refer the reader to \cite{Floyd}, \cite{Ka} for
more detail.

It is easy to see that $\rho_{v} (x, y) = \rho_{gv}(gx, gy)$ for any
$g \in G$. Floyd metrics relative to different basepoints are linked
by the following inequality,
\begin{equation}\label{equivmetric}
\lambda^{-d(v, v')} \le \frac{\rho_{v}(x, y)}{\rho_{v'}(x, y)} \le
\lambda^{d(v, v')}
\end{equation}
This implies that the action $G$ on $\Gx$ extends continuously to
$\pGf$.  In \cite{Ka}, Karlsson showed that $G$ acts on $\Gf$ as a
convergence group action. The lemma below that he uses to prove this
result is also very helpful in this paper. His lemma holds, in fact,
for any quasi-geodesic, but we only need the following simplified
version.
\begin{lem}[Visibility Lemma] \label{karlssonlem}\cite[Lemma 1]{Ka}
There is a function $\varphi: \mathbb R_{\ge 0} \to \mathbb R_{\ge 0}$ such that for
any $v \in G$ and any geodesic $\gamma$ in $\Gx$, if the following holds
$$\rho_v(\gamma_-, \gamma_+) \ge \kappa,$$ then $d(v,
\gamma) \le \varphi(\kappa)$. 
\end{lem}

The relevance of Floyd boundary to relative hyperbolicity is provided by the
following theorem by Gerasimov.
\begin{thm}\cite[Corollary 1.5]{Ge2}\label{Floyd}
Suppose $\GP$ is relatively hyperbolic with the Bowditch boundary
$\pG$. Then there exists $\lambda_0 >1$ such that there exists a
continuous $G$-equivariant surjective map
$$
\pGf \to \pG
$$
for any $\lambda > \lambda_0$.
\end{thm}

Using Theorem \ref{Floyd}, Gerasimov-Potyagailo gave a way to compactify $\Gx$ by $\pG$, and push the Floyd metric on $\Gf$ to get a shortcut metric on $\Gx\cup \pG$.

Let $\bar \rho_v$ be the maximal pseudo-metric on $\Gf$ such that
$\bar \rho_v(\cdot, \cdot) \le \rho_v(\cdot, \cdot)$ and $\bar
\rho_v$ vanishes on any pair of points in the preimage of a point
$\xi \in \pG$. Pushing forward $\bar \rho_v$, we can obtain a true
metric on $\Gx\cup \pG$, which is called \textit{shortcut
metric}. See \cite{GePo2} for more detail. 

Since $\bar \rho_v(\cdot, \cdot) \le \rho_v(\cdot, \cdot)$, Lemma \ref{karlssonlem} also holds for shortcut metrics $\bar \rho_v$ on $\Gx\cup \pG$. By abuse of language, we will denote by $\rho_v$ the corresponding shortcut metrics on $\Gx\cup \pG$. 

\subsection{Characterization of conical points}
In this subsection, we shall use the geometry of Cayley graph of $G$ to give a  characterization of conical points in Bowditch boundary $\pG$.   

\begin{lem}\label{charconical}
There exists a constant $r>0$ with the following property.

A point $\xi \in \pG$ is a conical point if and only if there exists a sequence of elements $g_n \in G$ such that for any
geodesic ray $\gamma$ originating at $\gamma_-\in G$ and ending at $\xi$, we have
$$\gamma \cap B(g_n, r) \neq
\emptyset$$ for all but finitely many $g_n$.
\end{lem}
\begin{proof}
The direction ``$=>$": Recall that a convergence action of $G$ on a compact space $M$ is called \textit{2-cocompact} if the action on $\{(x, y)\in M^2: x\ne y\}$ admits a compact fundamental domain. Since $G$ acts geometrically finitely on $\pG$, by a theorem of Gerasimov \cite{Ge1}, $G$ acts 2-cocompactly on $\pG$. By  \cite[Lemma 3.2]{GePo2}, the action $G$ on $G\cup\pG$ is also 2-cocompact. Then there exists a uniform constant $\epsilon
>0$ such that for any conical point $\xi \in \pG$, there exist $a, b \in \pG, g_n \in G$ such that $\rho_1(a, b) > \epsilon$ and $g_n(\xi, \zeta) \to (a, b)$ for $\zeta \in  G \cup \pG\setminus \xi$. Indeed, if $\rho_1(a, b) \le \epsilon$, there exists $g\in G$ by 2-cocompactness such that $\rho_1(ga, gb) \le \epsilon$. 

Define  $r =
\varphi(\epsilon/4)$, where $\varphi$ is given by Lemma \ref{karlssonlem}.  Consider a geodesic ray $\gamma$ with $\zeta :=\gamma_-$ in $G$ and $\xi:=\gamma_+\in \pG$. It thus follows
that $g_n \gamma \cap B(1, r) \neq \emptyset$, for all but finitely
many $g_n$. Hence, $\{g_n^{-1}\}$ is the sequence we are looking for.

The direction ``$<=$": Recall that, in \cite{GePo2}, a bi-infinite geodesic $\gamma$ in $\Gx$ is
called a \textit{horocycle} at $\xi \in \pG$ if the following holds
$$\lim_{t \to \infty} \gamma(t) = \lim_{t \to \infty} \gamma(-t) =
\xi.$$ Since $G$ acts geometrically finitely on $\pG$, $\xi$
is either bounded parabolic or conical. At parabolic points, there
always exist horocycles. Thus, each half of a horocycle gives raise to a geodesic ray but one of them  violates $\gamma \cap B(g_n, r) \neq
\emptyset$.  
\end{proof}
\begin{rem}
In the ``$=>$" direction the geometrically finite assumption is only needed to obtain a uniform constant $r$. Thus, the ``$=>$" direction holds for any conical point $\xi \in \pGf$, provided that we do not demand for a uniform constant $r$.  
\end{rem}

\subsection{Transition points}

There are only finitely many $G$-conjugacy
classes in $\mathcal P$, cf. \cite{Tukia}. We take a choice $\bar{\mathcal P}$ of a
complete set of conjugacy representatives in $\mathcal P$.  Define
$\mathbb P = \{gP: g \in G, P \in \mathcal{\bar P}\}$. It is
worth noting that $\mathbb P$ depends on the choice of
$\bar{\mathcal P}$. 

It is known in \cite{DruSapir}, \cite{GePo2} that $\mathbb P$ is a contracting system with bounded intersection in $\Gx$. Hence, the
notion of transition points relative to $\mathbb P$ applies in
$\Gx$. In fact, this notion was introduced by Hruska in \cite{Hru}.
It is further generalized and elaborated on via Floyd metrics by
Gerasimov-Potyagailo in \cite{GePo4}. Their viewpoint will be adopted
in this paper.

Following Gerasimov-Potyagailo,  we explain the role of transition
points via Floyd metrics in the following manner. Recall that the Bowditch
boundary $\pG$ is equipped with a family of shortcut metrics
$\{\rho_v\}_{v\in G}$.

\begin{lem}\label{floydtrans}\cite[Corollary 5.10]{GePo4}
There exists $\epsilon_0 >0$ with the following property.

For any $\epsilon\ge \epsilon_0, R > 0$, there exists $\kappa =
\kappa(\epsilon, R)>0$ such that for any geodesic $\gamma$ and an
$(\epsilon, R)$-transition point $v \in \gamma$, we have
$\rho_v(\gamma_-, \gamma_+) > \kappa$.
\end{lem}

In what follows, we make ourselves conform to the following
convention.

\begin{conv}[about  $\epsilon, R$]\label{epsilonR}
When talking about $(\epsilon, R)$-transition points in $\Gx$ we
always assume that $\epsilon \ge \epsilon_0$, where $\epsilon_0$ are
given by Lemmas \ref{uniformtrans} and
\ref{floydtrans}. In addition, assume that $R >R(\epsilon)$, where
$R(\epsilon)$ is given by Lemma \ref{entrytrans}.
\end{conv}

The following lemma is proven in \cite{Hru} by Hruska in the case
that $r=0$, with an alternative proof by Gerasimov-Potyagailo that can be found in
\cite{GePo4}. The proof below is inspired by
the one of Proposition 6.3 in \cite{GePo4}.

\begin{lem}\label{PairTransG}
Let $\epsilon, R$ be chosen as in Convention (\ref{epsilonR}). For any $r \ge 0$, there
exist $D=D(\epsilon, R)>0$ and $L=L(r, \epsilon, R)>0$ with the following property.

Let $\gamma, \gamma'$ be two geodesics in $\Gx$ such that
$\gamma_-=\gamma'_-$ and $d(\gamma_+, \gamma'_+) < r$. Take an $(\epsilon, R)$-transition point $v \in
\gamma$ such that $d(v,
\gamma_+) > L$. Then there exists an $(\epsilon, R)$-transition
point $w \in \gamma'$ such that $d(v, w) < D$.
\end{lem}
\begin{proof}
We first define our constants. Let $\kappa=\kappa(\epsilon, R)$ and $r_0
=\varphi(\kappa/2)$, where $\kappa, \varphi$ are given
by Lemmas \ref{floydtrans}, \ref{karlssonlem} respectively.  

Denote
$\kappa' =\lambda^{-r_0}\cdot \kappa/2$, where $\lambda>1$  satisfies Theorem \ref{Floyd}. By 
\cite[Proposition 4.1]{GePo4}, there exists $D_0=D(\epsilon, \kappa')>0$ such that the
following holds for any $Y \in \mathbb P$ and any $w\in G$,
\begin{equation}\label{closure}
N_\epsilon(Y) \setminus (B_{\rho_w}(q, \kappa'/6) \cap G) \subset
B(w, D_0),
\end{equation}
where $q$ is the one-singleton limit set of $Y$ in $\pG$.

Since $d(\alpha_+, \gamma_+)\le r$, we can choose $L > r$ large enough such that
$\rho_v(\gamma_+, \gamma'_+) < \kappa/2$.

Let $v \in \gamma$ be an $(\epsilon, R)$-transition point such that $d(v, \gamma_+) >L$. By Lemma \ref{floydtrans}, $\rho_v(\gamma_-,
\gamma_+) > \kappa$ and then $\rho_v(\gamma_-, \gamma'_+) >
\kappa/2$. So there exists $w \in \gamma'$
such that $d(v, w) < r_0$. By the inequality (\ref{equivmetric}), it follows that
$\rho_w(\gamma'_-, \gamma'_+) > \kappa'$.

Assume that $w$ is $(\epsilon, R)$-deep in some $V \in \mathbb P$,
otherwise we are done.  Let $x, y$ be the entry and exit points of
$\gamma'$ in $N_\epsilon(V)$ respectively. Then $x, y$ are $(\epsilon, R)$-transitional by Lemma \ref{entrytrans}.

\begin{claim}
$\min\{d(w, x),\;d(w, y)\} <D_0 + \varphi(\kappa'/3)$.
\end{claim}
\begin{proof}[Proof of Claim]
If $\rho_w(x, y) >
\kappa'/3$, then by (\ref{closure}), we have $\{x, y\} \cap B(w,
D_0) \neq \emptyset$.  Otherwise, we have $$\max\{\rho_w(\gamma'_-, x), \;\rho_w(y,
\gamma'_+)\} > \kappa'/3.$$ It follows by Lemma \ref{karlssonlem}
that $\min\{d(w, y),\;d(w, x)\} < \varphi(\kappa'/3)$.
\end{proof}

By the claim above, we have that one of $\{x, y\}$ is an $(\epsilon, R)$-transition point which has a distance at most $D= D_0+r_0+\varphi(\kappa'/3)$ to $v$. The proof is complete.
\end{proof}

\subsection{Horofunctions}
In this subsection, we shall define horofunctions cocycles at conical
points. We start with a horofunction associated to a geodesic ray.
\begin{defn}
Let $\gamma$ be an infinite geodesic ray (with length
parametrization) in $\Gx$. A \textit{horofunction $b_\gamma: G \to \mathbb
R$ associated to $\gamma$} is defined as follows
$$b_\gamma(x) = \lim\limits_{t \to \infty} (d(x, \gamma(t)-t),$$
for any $x \in G$.
\end{defn}

Let $\kappa>0$ be given by Lemma \ref{floydtrans}, $r>0$ given by Lemma
\ref{charconical} and $\varphi$ given by Lemma \ref{karlssonlem}. We
fix the constant $$C =\max\{ 2\varphi(\kappa/2), 4r\},$$ until the end of this
subsection.

\begin{lem}\label{horofunc}
Let $\xi \in \pG$ be a conical point at which terminate two
geodesics $\gamma, \gamma'$. Then the following holds
$$|(b_\gamma(x)-b_\gamma(y))-(b_{\gamma'}(x)-b_{\gamma'}(y)| < C, $$
for any $x, y \in G$.
\end{lem}
\begin{proof}
Note that $\xi$ is a conical point. By Lemma \ref{charconical}, there
exists a sequence of points $z_n \in G$ such that $d(z_n, \gamma),
d(z_n, \gamma') < r$.  Then we obtain two sequences of points $v_n
\in \gamma,\; v_n'\in \gamma'$ such that $d(v_n, v_n') \le 2r$. Thus,
$$
\begin{array}{lr}
\;\;\;\; |(b_\gamma(x)-b_\gamma(y))-(b_{\gamma'}(x)-b_{\gamma'}(y))| \\
 =\lim\limits_{n \to
\infty} |(d(v_n, x)-d(v_n, y))-(d(v_n', x)-d(v_n', y))| \\
\le 4r.
\end{array}
$$
The conclusion follows.
\end{proof}

We are now going to define a notion of a horofunction cocycle. Up to
some uniform constant, it is independent of the choice of geodesic rays .

Let $\xi \in \pG$ be a conical point. Let $B_\xi(\cdot,\cdot):
G\times G \to \mathbb R$ be defined as follows:
$$\forall x, y \in G: B_\xi(x, y) = \sup \limits_{\forall \gamma:
\gamma_+=\xi} \{b_\gamma(x) -b_\gamma(y)\}.$$ For any $z \in
G$, let $\forall x, y \in G: B_z(x, y) = d(z, x) - d(z, y).$ By
Lemma \ref{horofunc}, we see that $B_\xi(x, y)$ differs from
$$\forall \gamma, \gamma_+=\xi: \;b_\gamma(x) -b_\gamma(y)$$ by a
uniform constant depending only on $G$. In what follows, we usually
view $b_\gamma(x) -b_\gamma(y)$ as a representative of $B_\xi(x,y)$
and deal with it directly. In this case, we suppress the index
$\gamma$ for the convenience.

The following lemma is crucial in establishing the quasi-conformal
density on Bowditch boundary.

\begin{lem}\label{rephorofunc}
Let $\xi \in \pG$ be a conical point. For any $x, y \in G$, there is
a neighborhood $V$ of $\xi$ such that the following property holds:
$$|B_\xi(x, y)-B_z(x, y)| < 4C, \forall z \in V \cap G.$$
\end{lem}
\begin{proof}
Let $\gamma$ be a geodesic between $x$ and $\xi$. For simplicity, we let $B_\xi(x, y)=b_\gamma(x)-b_\gamma(y)$. Recall that $b_\gamma(x)$ is
the limit of a non-increasing function. Then for given $C
>0$, there exists a number $L>0$ such that
\begin{equation}\label{vlimit}
|b_\gamma(x)-b_\gamma(y) - (d(x, v)-d(y, v))| \le C
\end{equation}
for any $v\in \gamma$ with $d(x, v) >L$. 

Let $\epsilon, R$ be chosen as in Convention \ref{epsilonR}.  We
first note: 
\begin{claim}
There exists a sequence of $(\epsilon, R)$-transition
points in $\gamma$ tending to $\xi$. 
\end{claim}
\begin{proof}[Proof of the Claim] 
In fact, $\gamma$ exits the
$\epsilon$-neighborhood of every $Y \in \mathbb P$ which
it enters into. Let $u, v$ be the entry and exit
points of $\gamma$ in $N_\epsilon(Y)$ respectively for some $Y \in \mathbb P$ such that $d(u, x) \ge R$. If the middle
point $z$ of $[u, v]_\gamma$ is $(\epsilon, R)$-deep in $Y$, then
$u, v$ are both $(\epsilon, R)$-transition points in $\gamma$ by
Lemma \ref{entrytrans}. Otherwise, either $z$ is $(\epsilon,
R)$-transitional or $z$ is $(\epsilon, R)$-deep in some $Y' \neq Y
\in \mathbb P$. The former case already gives an $(\epsilon,
R)$-transition point $z$. In the latter case, the entry and exit
point of $\gamma$ in $N_\epsilon(Y')$ is $(\epsilon, R)$-transition
points in $\gamma$.  This proves our claim.
\end{proof}

Recall that $\rho_v(x, \xi) \ge \kappa$ for any $(\epsilon, R)$-transition point $v$ of $\gamma$ by Lemma
\ref{floydtrans}.  By the above claim, we choose $v \in \gamma$ such that $d(x, v) \ge L$,  and by Lemma
\ref{karlssonlem} such that \begin{equation}\label{rephorofuncE1}
\rho_v(x, y) \le \kappa/4.
\end{equation}

Let $\lambda>1$ be provided by Theorem \ref{Floyd}.  Consider $V = B_{\rho_1}(\xi, \lambda^{-d(1, v)}\kappa/4)$, which is the metric ball at $\xi$ radius $\lambda^{-d(1, v)}\kappa/4$ with respect to
the metric $\rho_1$. We shall show that $V$ is the desired
neighborhood.

For any $z \in V \cap G$, the inequality (\ref{equivmetric})
implies that $\rho_v(z, \xi) < \kappa/4$. We then obtain by (\ref{rephorofuncE1}) the following
\begin{equation}\label{rephorofuncE2}
\rho_v(y, z) \ge \kappa/2, \; \rho_v(x, z) > \kappa/2.
\end{equation}
Thus, by Lemma \ref{karlssonlem}, there exists $w
\in \gamma':=[y, z]$ such that $d(v, w) \le \varphi(\kappa/2)$. By (\ref{vlimit}), it follows that
$$\begin{array}{rl}
& \;\; |B_\xi(x, y)-(d(x, z)-d(y, z))|\\
& \le C + |(d(x,
v)-d(y, v))-(d(x, z)-d(y, z))| \\
&\le 2C + |(d(x, v) -d(x, z)) -(d(y, w) - d(y, z)) | +d(v, w) \\
&\le 2C + 2d(v, w) \le 4C.
\end{array}
$$
for $\forall z \in V \cap G$. This finishes the proof.
\end{proof}

\section{Visual and quasi-conformal densities}\label{Section3}
Suppose that $G$ acts properly and isometrically on a proper
geodesic metric space $(X, d)$. Assume, in addition, that $X$ admits
a compactification with at least 3 points, denoted by $\pG$, such that $X \cup \pG$ is a
compact metrizable space, and the action of $G$ on $X$ extends to a
minimal convergence group action on $\pG$. Denote by $\cG$ the set of conical limit points in $\pG$. Fix a basepoint $o \in X$.  

This section is to present a general machinery of the quasi-conformal density on $\pG$, which has the application to a relatively hyperbolic group $\GP$ with the topological boundary $\pG$.

\subsection{Quasi-conformal densities}
A Borel measure $\mu$ on a topological space $T$ is \textit{regular}
if $\mu(A) = \inf\{\mu(U): A \subset U, U$ is open$\}$ for any Borel
set $A$ in $T$. The $\mu$ is called \textit{tight} if $\mu(A) =
\sup\{\mu(K): K \subset A, K$ is compact$\}$ for any Borel set $A$
in $T$.

Recall that \textit{Radon} measures on a topological space $T$ are
finite, regular, tight and Borel measures. It is well-known that all
finite Borel measures on compact metric spaces are Radon. Denote by
$\mathcal M(\pG)$ the set of finite positive Radon measures on
$\pG$. Then $G$ possesses an action on $\mathcal M(\pG)$ given by
$g_*\mu(A) = \mu(g^{-1}A)$ for any Borel set $A$ in $\pG$.

Endow $\mathcal M(\pG)$ with the weak-convergence topology. Write
$\mu(f) = \int f d\mu$ for a continuous function $f \in C^1(\pG)$.
Then $\mu_n \to \mu$ for $\mu_n \in \mathcal M(\pG)$ if and only if $\mu_n(f)
\to \mu(f)$ for any $f \in C^1(\pG)$,  equivalently, if and only if,
$\liminf\limits_{n \to \infty}\mu_n(U) \ge \mu(U)$ for any open set
$U \subset \pG$.

\begin{defn}
Let $\sigma \in [0, \infty[$. A $G$-equivariant map $$\mu: G \to
\mathcal M(\pG),\; g \to \mu_g$$ is called a
\textit{$\sigma$-dimensional visual density} on $\pG$ if $\mu_g$ are
absolutely continuous with respect to each other and their
Radon-Nikodym derivatives satisfy
\begin{equation}\label{vdensity}
\exp(-\sigma d(go, ho)) \prec \frac{d\mu_{g}}{d\mu_{h}}(\xi)  \prec
\exp(\sigma d(go, ho)),
\end{equation}
for any $g, h \in G$ and  $\mu_h$-a.e. point $\xi \in \pG$.

Here $\mu$ is called \textit{$G$-equivariant} if $\mu_{hg}(A) = h_*\mu_{g}(A)$ for any Borel set $A \subset
\partial G$.
\end{defn}

\begin{rem}
The terminology ``visual density" is due to Paulin in \cite{Paulin}
to generalize Hausdorff measures of visual metrics. As $G$ acts minimally on $\pG$ with $|\pG|>3$, $G$ has no global fixed point on $\pG$. Then $\mu_g$
is not an atom measure.
\end{rem}

The following assumption is motivated by Lemma \ref{rephorofunc}. Given $z \in Go$, let $B_z(x, y) = d(x,
z) -d(y,z)$ for $x, y \in X$. 
\begin{assump}\label{AssumpA}
There exists a constant $C>0$ and a family of functions
$$\{B_\xi(\cdot, \cdot): Go \times Go \to \mathbb R\}_{\xi \in
\cG}$$ such that for any $x, y \in X$, there is a neighborhood $V$ of $\xi \in \cG$
in $X \cup \pG$ with the following property:
$$|B_\xi(x, y)-B_z(x, y)| < C, \forall z \in V \cap Go.$$
\end{assump}

The reader should keep in mind that the upper bound in
(\ref{vdensity}) is usually not applicable in practice. The following definition improves on
the upper bound.
\begin{defn}
Let $\sigma \in [0, \infty[$.  A $\sigma$-dimensional visual density
$$\mu: G \to \mathcal M(\pG), \;g \to \mu_g$$ is a
\textit{$\sigma$-dimensional quasi-conformal density} if for any $g,
h \in G$ the following holds
\begin{equation}\label{cdensity}
\frac{d\mu_{g}}{d\mu_{h}}(\xi) \asymp \exp(-\sigma B_\xi (go, ho)),
\end{equation}
for $\mu_h$-a.e. conical points $\xi \in \pG$.
\end{defn}

\begin{rem}
By \ref{AssumpA}, the function $B_\xi (\cdot, \cdot)$ was defined only for conical points $\xi$. So (\ref{cdensity}) is understood for almost every conical points $\xi \in \pG$ in the measure $\mu_h$.
\end{rem}

By the equivariant property of $\mu$, we see the following result.
\begin{lem}
Let $\{\mu_v\}_{v\in G}$ be a $\sigma$-dimensional quasi-conformal density on $\pG$.
Then the support of any $\mu_v$ is $\pG$.
\end{lem}

The notion of a (partial) shadow/cone is key to our study. See Examples \ref{contractingsystem}, for which the following convention is fulfilled. 

\begin{conv}\label{ConvContracting}
Let $\mathbb Y$ be a $G$-finite contracting system in $X$ with the bounded intersection property such that,   for each $Y\in \mathbb Y$, the stabilizer $G_Y$ acts co-compactly on either $Y$ or $\partial Y$. We consider below the transition points defined with respect to $\mathbb Y$.
\end{conv}


\begin{defn}[Shadow and Partial Shadow]
Let $r, \epsilon, R \ge 0$ and $g \in G$. The \textit{shadow $\Pi_r(go)$} at $go$ is the set of points
$\xi \in \partial G$ such that there exists SOME geodesic $\gamma=[o, \xi]$
intersecting $B(go, r)$.

The \textit{partial shadow $\Pi_{r, \epsilon,
R}(go)$} is the set of points $\xi \in \Pi_r(go)$ where, in addition, the geodesic $\gamma$ as above contains an $(\epsilon, R)$-transition point $v$ in $B(go, 2R)$.

The \textit{strong shadow $\varPi_r(go)$} at $go$
is the set of points $\xi \in \partial G$ such that for ANY geodesic
$[o,\xi]$ we have $[o,\xi] \cap B(go, r) \ne \emptyset$.
\end{defn}
\begin{rem}
The terminology ``strong shadow" corresponds to the usual notion of ``shadow'' in PS-measures,  see \cite{Coor} for example. It plays an essential role only in the proof of Proposition \ref{ballgrowth}.

\end{rem}

Inside the space $X$, the (partial) shadowed region motivates the notion of a (partial) cone.

\begin{defn}[Cone and Partial Cone]\label{defnpcone}
Let $g \in G$ and $r \ge 0$. The \textit{cone $\Omega_r(go)$} at $go$
is the set of elements $h$ in $G$ such that there exists SOME
geodesic $\gamma=[o, ho]$ in $X$ such that $\gamma \cap B(go, r) \neq
\emptyset$.

The \textit{partial cone
$\Omega_{r, \epsilon, R}(go)$} at $go$ is the set of elements $h \in
\Omega_r(go)$ such that one of the following statements holds.
\begin{enumerate}
\item $d(o, ho) \le d(o, go) + 2R,$
\item
The geodesic $\gamma$ as above contains an $(\epsilon,
R)$-transition point $v$ such that $d(v, go) \le 2R$.
\end{enumerate}
 
\end{defn}

For $r=0$, we omit the index $r$ and write $\Omega(go), \Omega_{\epsilon, R}(go)$
for simplicity.

For $\Delta \ge 0, n \ge 0$, define
$$
A(go, n, \Delta) = \{h \in G:  n-\Delta \le d(o, ho) - d(o, go) <
n+\Delta\},
$$
for any $g \in G$. 
\subsection{Patterson-Sullivan measures}
The aim of this subsection is to recall a construction of Patterson
and to show that Patterson-Sullivan measure is a quasi-conformal
density.


We associate the Poincar\'e series to
$Go$:
$$\PS{G} = \sum\limits_{g \in G} \exp(-sd(o, go)), \; s \ge 0,$$
for which the \textit{critical exponent} of $\PS{G}$ is given by
$$\g G = \limsup\limits_{R \to \infty} \log \sharp N(o, R)/R.$$
The group $G$ is of \textit{divergent} (resp.
\textit{convergent}) type with respect to $d$ if $\PS{G}$ is divergent (resp. convergent) at
$s=\g G$. It is clear that the definition does not depend on the choice of $o$.

We start by constructing a family of measures $\{\mu_v^s\}_{v\in G}$ supported on $Go$ for any $s >\g G$. First,
assume that $\PS{G}$ is divergent at $s=\g G$. Set
$$\mu^s_v = \frac{1}{\PS{G}} \sum\limits_{g \in G} \exp(-sd(vo, go)) \cdot \dirac{go},$$
where $s >\g G$ and $v \in G$. Note that $\mu^s_1$ is a probability
measure.

If $\PS{G}$ is convergent at $s=\g G$, Patterson introduced in
\cite{Patt} a monotonically increasing function $H: \mathbb R_{\ge 0} \to \mathbb R_{\ge 0}$ with the
following property:
\begin{equation}\label{patterson}
\forall \epsilon >0, \exists t_\epsilon, \forall t
> t_\epsilon, \forall a>0: H(a+t) \le \exp(a\epsilon) H(t).
\end{equation}
such that the series $\PSS{G}:=\sum\limits_{g \in G} H(d(vo, go)) \exp(-sd(v,
go))$ is divergent for $s \le \g G$ and convergent for $s>\g G$.
Then define measures as follows:
$$\mu^s_v = \frac{1}{\PSS{G}} \sum\limits_{g \in G} \exp(-sd(vo, go)) H(d(vo, go)) \cdot \dirac{go},$$
where $s >\g G$ and $v \in G$.

Choose $s_i \to \g G$ such that $\mu_v^{s_i}$ are convergent in
$\mathcal M(\pG)$. Let $\mu_v = \lim \mu_v^{s_i}$ be the limit
measures, which are so called \textit{Patterson-Sullivan measures}. Note that $\mu_1(\pG) = 1$. In the sequel, we write PS-measures as shorthand for
Patterson-Sullivan measures.

\begin{prop}\label{prequasiconf}
PS-measures are $\g G$-dimensional quasi-conformal densities.
\end{prop}
\begin{proof}
With \ref{AssumpA}, the proof goes exactly as that of Th\'eor\`eme
5.4 in \cite{Coor}.
\end{proof}
In Sections \ref{Section4} \& \ref{Section5}, we consider the group action of $G$ on its Cayley graph $\Gx$, which is compactified by $\pG$ such that $\Gx\cup \pG$ is endowed with the quotient topology of $\Gf$. 
By Lemma \ref{rephorofunc}, \ref{AssumpA} holds in this setting and thus Proposition \ref{prequasiconf} applies. 

In Section \ref{Section7}, we consider a cusp-uniform action of $G$ on a hyperbolic space $X$, which is compactified by $\pX=\pG$. In this setting Proposition \ref{prequasiconf} was already known in \cite{Coor}.


\section{Patterson-Sullivan measures on Bowditch boundary}\label{Section4}

In this section, we consider the PS-measures constructed on the
completion $\bar G:=\Gx \cup \pG$, where $\pG$ is the Bowditch boundary of
$G$ with shortcut metrics. In this case, we set the basepoint $o=1$. The generalities of Section \ref{Section3} shall get
specialized and simplified in this context with aid of the following
two facts:

\begin{enumerate}
\item
$G$ is divergent with respect to the word metric $d$.
\item
The Karlsson Visibility Lemma \ref{karlssonlem} holds for the
compactification $G \cup \pG$.
\end{enumerate}

\subsection{Shadow Lemma}
The key observation in the theory of PS-measures is the Sullivan
Shadow Lemma concerning the relation between the measure on boundary and
geometric properties inside.

We begin with some weaker forms of the Shadow Lemma (under weaker
assumption), with proofs following closely the presentation of
Coornaert in \cite[Section 6]{Coor}.


\begin{lem}[Shadow Lemma I]\label{weakombreI}
Let $\{\mu_v\}_{v \in G}$ be a $\sigma$-dimensional visual density on $\pG$. Then
there exists $r_0 > 0$ such that the following holds
$$
\exp(-\sigma d(1, g)) \prec \mu_1(\varPi_r(g)) \le
\mu_1(\Pi_r(g)),
$$
for any $r\ge r_0$ and $g \in G$.
\end{lem}

\begin{proof}
Let $m_0$ be the maximal value over atoms of $\mu_1$. Since $\mu_1$
is not an atom measure, it follows that $m_0 < \mu_1(\partial G)$. We consider the shortcut metric $\rho_1$ on $\pG$ in this proof.

Fix $m_0 < m < \mu_1(\partial G)$. There is a constant
$\epsilon_0 > 0$ such that any subset of diameter at most $\epsilon_0$
has a measure at most $m$. Indeed, if not, then there are a sequence
of positive numbers $\epsilon_n \to 0$ and subsets $X_n$ of diameter
at most $\epsilon_n$ such that $\mu_1(X_n) > m$. Then up to passage
of a subsequence,  $\{X_n\}$ converges to a point $p \in \pG$. Since
$\mu_1$ is regular, we have $\mu_1(p) = \inf\{\mu_1(U)\}$, where the
infimum is taken over all open sets $p \in U$. Thus, $\lim_{n\to \infty}
\mu_1(X_n) = \mu_1(p)$. This contradicts to the choice of $m$.

The following fact is a consequence of Visibility Lemma
\ref{karlssonlem}. 

\begin{claim}\label{bdcomp}
Given any $\epsilon > 0$, there is a constant $r_0 > 0$ such that
the following holds $$\diam{\pG\setminus g^{-1} \varPi_r(g)} <
\epsilon$$ for all $g \in G$ and $r > r_0$.
\end{claim}
\begin{proof}[Proof of Claim]
Given $\epsilon > 0$, let $r_0 = \varphi(\epsilon/2)$ be the
constant provided by Lemma \ref{karlssonlem}. Assume that $d(g, 1)
> r >r_0$. Otherwise, $\varPi_r(g)=\pG$ and there is nothing to do.

Observe that $g^{-1} \varPi_r(g)$ is the set of boundary points $\xi
\in\pG$ such that any geodesic $[g^{-1}, \xi]$ passes through the
closed ball $B(1, r)$. Let $\xi \in \pG \setminus g^{-1} \varPi_r(g)$.
Thus, some geodesic $[g^{-1}, \xi]$ misses the ball $B(1, r)$. Hence,
$\rho_1(g^{-1}, \xi) \le \epsilon/2$. This implies that $\diam{\pG
\setminus g^{-1} \varPi_r(g)} <\epsilon$.
\end{proof}
\begin{rem}
The same claim holds for $\pG\setminus g^{-1} \Pi_r(g)$ by the
same argument.
\end{rem}

Hence, for $\epsilon=\epsilon_0$, there is a constant $r_0>0$ given
by the Claim such that the following holds
$$\mu_1(\pG) - m < \mu_1( g^{-1}\varPi_r(g)),\; \forall r > r_0,$$
for all $g \in G$.

Since $\mu $ is a $\sigma$-dimensional visual density, there is a
constant $C>0$ such that the following holds
$$
C^{-1} \exp(-\sigma d(1, g)) \le
\frac{\mu_{g^{-1}}(g^{-1}\varPi_r(g))}{\mu_1( g^{-1}\varPi_r(g)) }.
$$
By the equivariant property of $\mu$, it follows that
$$\mu_{g^{-1}}(g^{-1}\varPi_r(g)) = \mu_1(\varPi_r(g)).$$ This
implies that
$$
(\mu_1(\pG)-m) C^{-1} \exp(-\sigma d(1, g)) \le
\mu_1(\varPi_r(g)),
$$
for all $g \in G$. This concludes the
proof.
\end{proof}

Let's denote by $\Pi^c_r(g)$ the set of all the conical points in
$\Pi_r(g)$.
\begin{lem}[Shadow Lemma II]\label{weakombreII}
Let $\{\mu_v\}_{v\in G}$ be a $\sigma$-dimensional quasi-conformal density on $\pG$.
Then there exists $r_0 > 0$ such that the following holds
$$
\mu_1(\Pi^c_r(g)) \prec \exp(-\sigma d(1,
g)) \cdot \exp(2\sigma r),
$$
for any $g\in G$ and $r > r_0$.
\end{lem}
\begin{proof}
Let $\xi \in g^{-1} \Pi^c_r(g)$. Then there is a geodesic
$\gamma$ between $g^{-1}$ and $\xi$ such that $\gamma \cap B(1, r)
\neq \emptyset$. Hence, the following holds $$|d(z, g^{-1})-d(z, 1)
- d(g^{-1}, 1)| < 2r,$$ for any $z \in \gamma$ with $d(1, z) \ge d(1, g) +2r$.

By Lemma \ref{horofunc}, there is a uniform constant $C_1$ such that $$|B_\xi(g^{-1}, 1)
- \lim\limits_{t \to \infty} (d(\gamma(t),
g^{-1})-d(\gamma(t), 1))| \le C_1.$$

Since $\mu$ is a $\sigma$-dimensional quasi-conformal density, there
is a constant $C_2>0$ such that the following holds
$$
C_2^{-1} \exp(-\sigma B_\xi(g^{-1}, 1)) \le
\frac{d\mu_{g^{-1}}}{d\mu_1}(\xi) < C_2 \exp(-\sigma B_\xi(g^{-1},
1))
$$
for $\mu_1$-a.e. conical points $\xi \in \cG$.

By the equivariant property of $\{\mu_v\}_{v\in G}$, it follows that
$$\mu_{g^{-1}}(g^{-1}\Pi_r(g)) = \mu_1(\Pi_r(g)).$$

Observe that the shadow set $\Pi_r(g)$ is closed and thus
$\Pi^c_r(g)$ is Borel. We use Radon-Nikodym derivative to derive
the following
$$
\begin{array}{rl}
\mu_1(\Pi^c_r(g)) & \le \int_{g^{-1}\Pi^c_r(g)} C_2\exp(-\sigma B_\xi(g^{-1}, 1)) d \mu_1(\xi)  \\
& \prec C_2 \exp(C_1\sigma) \exp(-\sigma d(g^{-1}, 1)+2r\sigma)
\mu_1(g^{-1}\Pi^c_r(g)),
\end{array}
$$
which finishes the proof by the fact $\mu_1( g^{-1}\Pi_r(g)) \le 
\mu_1(\pG) =1$.
\end{proof}

Combining Shadow Lemma (I) \ref{weakombreI} and (II)
\ref{weakombreII}, we obtain the full strength of Shadow Lemma for a
$\sigma$-dimensional quasi-conformal density \textbf{without atoms at
parabolic points}.
\begin{lem}[Shadow Lemma]\label{ShadowLem}
Let $\{\mu_v\}_{v\in G}$ be a $\sigma$-dimensional quasi-conformal density without
atoms at parabolic points. Then there is $r_0 > 0$ such that the
following holds
$$
\begin{array}{rl}
\exp(-\g G d(1, g)) \prec \mu_1(\varPi_r(g))  \le \mu_1(\Pi_r(g)) \prec_r \exp(-\g G d(1, g))\\
\end{array}
$$
for any $g\in G$ and $r > r_0$.
\end{lem}

\subsection{First applications to growth functions}
In this subsection, we shall see how the growth rate of a relatively
hyperbolic group is related to the dimension of visual and quasi-conformal
density.
\begin{prop}\label{ballgrowth}
Let $\{\mu_v\}_{v\in G}$ be a $\sigma$-dimensional visual density on $\pG$. Then
the following holds $$\sharp N(1, n) \prec \exp(\sigma n),$$ for any
$n\ge 0$.
\end{prop}
\begin{proof}
Denote by $S^k$ the set of elements $g$
in $G$ with $d(g, 1) = k$. Let $r, C_1$ be constants given by Lemma \ref{weakombreI} such that $$\mu_1(\varPi_r(g)) > C_1
\exp(-\sigma d(1, g)) > C_1\exp(-\sigma k)$$ for any $g \in S^k$.

Observe that each point of $\pG$ is contained, if at all, in at most
$C_2$ sets of form $\varPi_r(g)$, where $C_2>1$ depends only on $r$.
Indeed, given $\xi \in \pG$, let $g$ be such that $\xi \in
\Pi_r(g)$. Fix a geodesic $\gamma=[1, \xi]$ and a point $x \in
\gamma$ such that $d(1, x)= k$. By the definition of strong shadow, we have
$d(g, \gamma) < r$. Then $g \in B(x, 2r)$.  Hence, $C_2 = \sharp
B(x, 2r)$ gives the desired number.

Hence, for each $k > 0$, we have $\bigcup_{g \in S^k}
\mu_1(\varPi_r(g))  < C_2 \mu_1(\pG)$.  Then we obtain
$$\sharp{S^k} < C_3 \exp(\sigma k),$$ where $C_3 = C_1^{-1}C_2
\mu_1(\pG)$. Consequently, the following holds
$$\sharp N(1, n) = \sum\limits_{k=0}^{n} \sharp S^k < C
\exp(\sigma n),$$ for some constant $C \ge 1$.
\end{proof}
\begin{rem}
In the proof,  the strong shadow $\varPi_r(g)$ is necessary to
obtain the uniform constant $C_2$.
\end{rem}

We obtain a few corollaries as follows.
\begin{cor}\label{visualdimbd}
Let $\mu$ be a $\sigma$-dimensional visual density on $\pG$. Then
$\sigma \ge \g G$.
\end{cor}

\begin{cor}\label{ballgrowthI}
The following holds $$\sharp N(1, n) \prec \exp(\g G n),$$ for any
$n\ge 0$.
\end{cor}
\begin{proof}
By Proposition \ref{prequasiconf}, PS-measures $\{\mu_v\}_{v\in G}$ is a $\g G$-dimensional
visual density.
\end{proof}

Let's list $G = \{g_1, \ldots, g_i, \ldots \}$. By Lemma
\ref{charconical}, we have $\cG = A_r$ for any $r \gg 0$, where
\begin{equation}\label{conicalset}
A_r = \bigcap_{n=1}^{\infty} \bigcup_{i=n}^{\infty}
\Pi_r(g_i). 
\end{equation} In
other words, a conical point is shadowed infinitely many times
by elements in $G$.

The observation (\ref{conicalset}) can be used to impose an upper
bound on the dimension of an arbitrary quasi-conformal density on
$\pG$ by the growth rate of $G$.

\begin{lem}
Let $\{\mu_v\}_{v\in G}$ be a $\sigma$-dimensional quasi-conformal density on $\pG$.
If for some (thus any $v\in G$) $\mu_v$ gives positive measure to the set of conical points, then
$\sigma \le \g G$.
\end{lem}
\begin{proof}
List $G = \{g_1, \ldots, g_i, \ldots\}$ such that $d(1, g_i)\le d(1, g_{i+1})$. We fix $r> r_0$, where  $r_0$ is given by Lemma
\ref{weakombreII} such that  $\mu_1(\cG)\ge \mu_1(A_r) > 0$.
Recall that
$$\Ps{G} = \sum\limits_{g \in G} \exp(-sd(1, g)) = \sum\limits_{k=0}^{\infty} \sharp S^k \exp(-sk).$$

We claim that $\Ps{G}$ is divergent at $s=\sigma$. Indeed, by Lemma
\ref{weakombreII}, we see that  
$$
\begin{array}{rl}
\sum\limits_{k=n}^{\infty} \sharp S^k \exp(-\sigma k) & \succ_r \sum\limits_{d(g, 1)\ge n} \mu_1(\Pi^c_r(g)) \\
&\succ_r \mu_1(\bigcup\limits_{d(g, 1)\ge n} \Pi^c_r(g))\\
& \succ_r \mu_1(A_r) > 0
\end{array}
$$
for all $n > 0$. Since $\mu_1(A_r)$ is independent of $n$, we have $\Ps{G}$ is divergent at
$s=\sigma$. This completes the proof.
\end{proof}
Corollary \ref{visualdimbd} implies:
\begin{cor}
If $\{\mu_v\}_{v\in G}$ is a $\sigma$-dimensional quasi-conformal density for $G$
giving positive measure on conical points, then $\sigma=\g G$.
\end{cor}

\subsection{PS-measures are quasi-conformal density without atoms}
Let $\{\mu_v\}_{v\in G}$ be PS-measures constructed on $\Gx \cup \pG$. We first show that $\mu_v$ has no atoms at parabolic points,
following an argument of Dal'bo-Otal-Peign\'e in \cite[Propositions
1 \& 2]{DOP}. A similar result as \cite[Proposition 1]{DOP} was also independently obtained by Matsuzaki in \cite[Lemma 30]{Matsu} for Kleinian groups. \footnote{The author is grateful to Professor K. Matsuzaki for communicating this result to him. }

Recall that the limit set $\Lambda(H)$ of a subgroup $H$ is the set of accumulation points of $H$ in the compactification $\Gx \cup \pG$. See Section \ref{Section2} for more details.
\begin{lem}\label{convpara}
Let $H$ be a subgroup in $G$ such that $\Lambda(H)$ is properly
contained in $\pG$. Then for any $v\in G$, $\PSv{H}$ is convergent at $s=\g G$. In
particular, if $H$ is divergent, then $\g H < \g G$.
\end{lem}
\begin{proof}
Since $H$ acts properly on $\pG \setminus \Lambda(H)$, choose  a
Borel fundamental domain $K$ for this action such that the
following holds
$$
\mu_1(\pG) = \sum\limits_{h \in H} \mu_1(hK) + \mu_1(\Lambda(H)).
$$
Observe that $\mu_1(K)>0$. Indeed, if not, we see that $\mu_1$ is
supported on $\Lambda(H)$. This gives a contradiction, as $\mu_1$ is
supported on $\pG$. Thus, $\mu_1(K) >0$.

Note that $\{\mu_v\}_{v\in G}$ is a $\g G$-dimensional visual density. It follows
that we have $$\mu_1(h^{-1}K)  = \mu_{h}(K) \ge \exp(-\g G d(1,
h)) \mu_1(K)$$ for any $h \in H$. Hence,
$$
\mu_1(\pG) \ge \sum\limits_{h \in H} \mu_1(hK)  \ge \sum\limits_{h
\in H} \exp(-\g G d(1, h)) \mu_1(K),
$$
implying that $\sum\limits_{h \in H}\exp(-\g G d(1, h))$ is finite.
This concludes the proof.
\end{proof}

\begin{lem} \label{parabnoatom}
The PS-measures $\{\mu_v\}_{v\in G}$ have no atoms at bounded parabolic points.
\end{lem}
\begin{proof}
Let $q \in \pG$ be a bounded parabolic point, and $P$ be the stabilizer of $q$ in
$G$. Choose a compact fundamental domain $K$ for the action
of $P$ on $(\pG\cup G) \setminus q$. Enlarge $K$, if necessary, in $G \cup \pG$ such that $q \notin K$ and the boundary of $K$ is
$\mu_v$-null for some (hence any) $v\in G$.

Since the closures of $P$ and $K$ at infinity are disjoint in $\pG$, we have by \cite[Proposition 3.3]{GePo4} that  $\proj_P(K\cap G)$ is
a finite set.

We choose the basepoint $v \in \proj_P(K\cap G)$.  By Lemma
\ref{contracting}, there is a constant $M>0$ such that the following
holds
\begin{equation}
d(z, v) + M \ge d(z, pv) + d(v, pv) \ge d(v, z),
\end{equation}
for any $z \in pK$ and $p\in P$.

Define $V_n = \bigcup\limits_{d(v, pv) > n}^{p\in P} p K$. Then $V_n \cup q$
is a decreasing sequence of open neighborhoods of $q$. Note that the
boundary of $V_n$ is $\mu_1$-null. It follows that $\mu_1^{s}(V_n)
\to \mu_1(V_n)$ for each $V_n$, as $s \to \g G$. 

Since $G$ is of divergent type for word metrics,  we have
$$
\begin{array}{rl}
\mu_v^s(V_n) & = \mu_v^s(\bigcup\limits_{d(v, pv) > n}^{p\in P} p K) \le \sum\limits_{d(v, pv) > n} \mu_v^s(p K)  \\
& \le \frac{1}{\PSv{G}}\sum\limits_{d(v, pv) > n}^{p\in P}\left(
\sum\limits_{z \in pK}\exp(-sd(v, z)) \cdot \dirac{z}\right)\\
& \le \frac{\exp(sM)}{\PSv{G}} \sum\limits_{d(v, pv) > n}^{p\in P}
\left(\sum\limits_{z \in pK}\exp(-sd(pv, z)-sd(v, pv)) \cdot \dirac{z}\right)
\end{array}
$$
yielding
$$
\begin{array}{rl}
\mu_v^s(V_n) & \le \frac{\exp(sM)}{\PSv{G}}
\sum\limits_{d(v, pv) > n}^{p\in P} \exp(-sd(v,pv)) \left(\sum\limits_{z \in K}\exp(-sd(v, z)) \cdot\dirac {z}\right)\\
& \le \exp(sM)\cdot \mu_v^s(K)\cdot \sum\limits_{d(v, pv) > n}^{p\in P}  \exp(-sd(v, pv)).
\end{array}
$$

Note that $\PSv{P}$ is convergent at $s=\g G$. Let $s \to \g G$ and
then $n\to \infty$. We obtain that $\mu_v(q) =0$. 
\end{proof}

We next consider conical points. In what follows until the end of
this subsection, we fix $r_0 = \varphi(\kappa/2)$, where
$\kappa=\kappa(\epsilon, R)>0$ is given by Lemma \ref{floydtrans} and $\varphi$ given by
Lemma \ref{karlssonlem}. Let $\lambda >1$ be given by Theorem \ref{Floyd}.

\begin{lem}\label{diskinombre}
Let $\xi \in \pG$ be a conical limit point and $\gamma$ be a
geodesic between $1$ and $\xi$. Let $v$ be an $(\epsilon, R)$-transition point on
$\gamma$. Denote $C_1=\kappa/2$ and $C_2= 2(\frac{1}{1-\lambda^{-1}}+r_0\lambda^{2r_0})$. Then
$$B_{\rho_1}(\xi, C_1 \lambda^{-d(v,1)}) \subset \Pi_{r_0}(v) \subset B_{\rho_1}(\xi, C_2\lambda^{-d(v,1)}).$$
\end{lem}
\begin{proof}
Let $\eta \in B_{\rho_1}(\xi, \frac{\kappa}{2\lambda^{d(v,1)}})$. By
the property (\ref{equivmetric}), it follows that $\rho_v(\eta, \xi)
\le \lambda^{d(v,1)} \rho_1(\eta, \xi) \le \kappa/2$. Hence, $\rho_v(1, \eta) \ge \kappa/2$ and then any
geodesic between $1, \eta$ has to pass through the ball $B(v, r_0)$. This proves the first inclusion.

Let $\eta \in \Pi_{r_0}(v)$. Then $d(v, [1, \eta]) \le r_0$ for some geodesic $[1, \eta]$. Consequently, there exists $w\in [1, \eta]$ such that $d(1, w)=d(1, v)$ and $d(v, w) \le 2r_0$. Observe that any segment $q$ of $[1, \xi]$ is a Floyd geodesic with respect to $\rho_1$. Thus, we obtain by calculation
$$
\rho_1(\xi, \eta) \le \rho_1(v, \xi)+\rho_1(w, \eta) + \rho_1(v, w) \le (\frac{2}{1-\lambda^{-1}}+2r_0\lambda^{2r_0}){\lambda^{-d(v,1)}}.
$$
This completes the proof.
\end{proof}

We are able to estimate PS-measures of a sequence of shrinking
$\rho_1$-metric disks centered at a conical point.

\begin{lem}\label{ombredisks}
Let $\{\mu_v\}_{v\in G}$ be a $\sigma$-dimensional quasi-conformal density without
atoms at parabolic points on $\pG$.  Then there exist constants $C_1, C_2
>0$ with the following property.

For any conical limit point $\xi \in \pG$, there exists a sequence of decreasing
numbers $\{r_i>0: r_i \to 0\}$ such that the following holds
$$
C_1 r_i^{-\sigma/\log
\lambda} \le \mu_1(B_{\rho_1}(\xi, r_i)) \le C_2
r_i^{-\sigma/\log
\lambda},$$ for all $i >0$.
\end{lem}
\begin{proof}
Observe that there are infinitely many $(\epsilon, R)$-transition points $v$ on
$\gamma:=[1, \xi]$.  (See the claim in the proof of Lemma \ref{rephorofunc}.)

By Lemmas \ref{diskinombre}, \ref{weakombreII}, there are $C_1, C_2>0$ such that 
\begin{equation}\label{ballestimate}
C_1 r^{-\sigma/\log
\lambda} \le \mu_1(B_{\rho_1}(\xi, r)) \le C_2
r^{-\sigma/\log
\lambda},
\end{equation}
where $r := \lambda^{-d(v,1)}$. Letting $d(v, 1) \to \infty$ gives a sequence $\{r_i
>0\}$ such that (\ref{ballestimate}) holds. This completes the proof.
\end{proof}

As a corollary, we obtain the following result.
\begin{lem}\label{conicnoatom}
The PS-measures have no atoms at conical points.
\end{lem}

As $G$ acts  geometrically finitely on $\pG$, there exist only bounded
parabolic points and conical points in $\pG$. Hence, Lemmas
\ref{parabnoatom} and \ref{conicnoatom} together prove the
following proposition.
\begin{prop}\label{PSatom}
Any PS-measure constructed on $\Gx \cup \pG$ is a $\g G$-dimensional
quasi-conformal density without atoms.
\end{prop}

In Appendix \ref{appendixA}, we further prove that PS-measures are unique and ergodic.

\section{Applications towards growth of cones and partial cones}\label{Section5}

This section applies the results established in Section \ref{Section4} to the study of the growth of cones in the Cayley graph $\Gx$. Let  $\{\mu_v\}_{v\in G}$ be PS-measures constructed through the action of $G$ on $\Gx\cup \pG$. The main tool is the partial Shadow Lemma that we shall prove in Subsection \ref{partialshadowlem}. 

We fix the basepoint $o=1$. For $r, \epsilon, R, \Delta > 0$, define
$$
\Omega_r(g, n, \Delta) = \Omega_r(g) \cap A(g, n, \Delta)
$$
and
$$
\Omega_{r, \epsilon, R}(g, n, \Delta) = \Omega_{r, \epsilon, R}(g)
\cap A(g, n, \Delta),
$$
for any $g \in G, n \ge 0$. 

\subsection{Growth of Cones}

We first consider the growth of cones. The proof of Lemma
\ref{ConeGrowthG} is simple, but illustrates the logic theme
of the proofs for more complicated ones, see Propositions \ref{PConeGrowthG} and
\ref{PConeGrowthX}.

\begin{lem}\label{ConeGrowthG}
There are constants $r, \Delta, \kappa>0$ such that the following
holds
$$
\sharp {\Omega_r(g, n, \Delta)} \ge \kappa \cdot \exp(\g G n)
$$
for any $g \in G$ and $n \ge 0$.
\end{lem}
\begin{proof}
Fix any $\Delta>0$. Let $r_0$ be the constant given by Shadow Lemma \ref{ShadowLem}. We
first note the following observation, roughly saying that the shadow
of a ball at $g$ can be covered by those of all elements at level
$n$ in the cone at $g$.

\begin{claim}
The following holds
$$
\mu_1(\Pi_{r}(g)) \le \sum\limits_{h \in \Omega_{r}(g,n, \Delta)}
\mu_1(\Pi_{r}(h)),
$$
for any $r >r_0$.
\end{claim}
\begin{proof}[Proof of Claim]
Let $\xi \in \Pi_r(g)$. Then there exists some geodesic $\gamma=[1,
\xi]$ such that $\gamma \cap B(g, r) \neq \emptyset$. Take $h \in
\gamma$ such that $d(h, 1)=n + d(1, g)$. Then $h \in \Omega_r(g,
n, 0)$. Hence, the conclusion follows.
\end{proof}

By Shadow Lemma \ref{ShadowLem}, it follows that
$$
\exp(n\g G) \cdot \mu_1(\Pi_r(h))  \asymp_{r, \Delta}
\mu_1(\Pi_{r}(g)),
$$
for any $h \in \Omega_{r}(g,n, \Delta)$. Hence the conclusion follows by the Claim above.
\end{proof}

\subsection{Partial Shadow Lemma}\label{partialshadowlem}
Let $\mathbb Y$ be a $G$-finite contracting system with the bounded intersection property in $\Gx$ such that $\mathbb P\subset \mathbb Y$ and, for each $Y\in \mathbb Y$, the stabilizer $G_Y$ acts co-compactly on $Y$. The partial cone we consider is defined using transition points relative to $\mathbb Y$. See Convention \ref{ConvContracting}. 

To study the growth of partial cones, the key tool is a variant of
the Shadow Lemma for partial shadows, which we shall call \textit{Partial
Shadow Lemma}.  
\begin{lem}[Partial Shadow Lemma]\label{PShadowLem}
Let $\epsilon$ be as in Convention \ref{epsilonR}. There exist
constants $r, R
>0$ such the following holds
$$
\exp(-\g G d(1, g)) \prec \mu_1(\Pi_{r, \epsilon, R}(g)),
$$
for any $g \in G$.
\end{lem}
\begin{rem}
We also prove  an analogue for   a cusp-uniform action of $G$ on $X$, cf. 
\cite[Lemma 3.18]{YANG8}.
\end{rem}

The following observation is a crucial fact in the proof.
\begin{lem}\label{convps}
For any $\varepsilon, r >0$, there exists $R=R(\varepsilon, r)>0$ such
that the following holds
$$
\sum \limits_{z \in Y}^{d(z, v) > R} \exp(-\g G d(v, z)) <
\varepsilon,
$$
for any $Y \in \mathbb Y$ and $v \in N_r(Y)$.
\end{lem}
\begin{proof}
Since $\mathbb P\subset \mathbb Y$ has the bounded intersection property, we see that the closure of each $Y\in \mathbb Y$ in $\pG$ is a proper subset so that the limit set of the stabilizer $G_Y$ of $Y$ in $G$ is also proper.   By Lemma \ref{convpara}, $\PSv{G_Y}$ is convergent at
$s=\g G$.  Since $G_Y$ acts co-compactly on $Y$, the lemma thus follows.
\end{proof}

\begin{proof}[Proof of Lemma \ref{PShadowLem}]
Given $g\in G$, there exist $r, C_1, C_2>0$ by Shadow Lemma \ref{ShadowLem} such that
\begin{equation}\label{LBND}
C_1 \exp(-\g G d(1, g)) \le \mu_1(\Pi_r(g)) \le C_2 \exp(-\g G d(1, g)).
\end{equation} 
Denote by $\mathbb F$ the set of $Y \in \mathbb Y$ such
that $Y \cap B(g, r + \epsilon) \ne \emptyset$. Then $\sharp \mathbb F$ is
a uniform number depending only on $G$.

By Lemma \ref{convps}, there exists $R_0>0$ depending on $r, \epsilon$ and $\sharp \mathbb F$ such that
\begin{equation}\label{SUM}
\begin{array}{rl}
\sum \limits_{z \in N_\epsilon(Y)}^{d(z, g) > R_0} \exp(-\g G d(g, z)) \cdot \sharp \mathbb F \cdot \exp(2\g G r) <
C_1/(2C_2)
\end{array}
\end{equation}
for any $Y \in \mathbb F$.

Let $R_1>0$ be given by Lemma \ref{entrytrans}.  Choose $R=r +
\max\{R_0, R_1\}$. In the remainder of the proof, we show that $r, R$ are
the desired constants.

Denote $\Xi := \Pi^c_r(g) \setminus \Pi_{r, \epsilon, R}(g)$. By definition of the partial shadow, for any given $\xi \in \Xi$, there exists a geodesic $\gamma=[1, \xi]$ such that $\gamma\cap B(g, r)\ne \emptyset$ and $\gamma$ does not contain an $(\epsilon,
R)$-transition point in the ball $B(g, 2R)$.

\begin{figure}[htb]
\centering \scalebox{0.4}{
\includegraphics{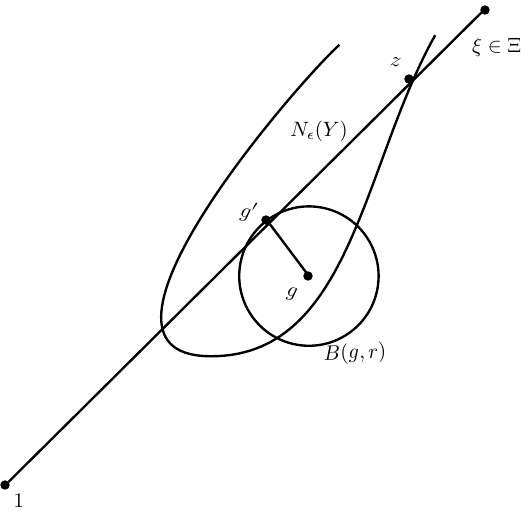} 
} \caption{Lemma \ref{PShadowLem}} \label{Figure1}
\end{figure}

Choose $g' \in B(g, r) \cap \gamma$ such that $d(g, g') < r$. Thus,
$\gamma$ does not contain an $(\epsilon, R)$-transition point in
$B(g', 2R-r)$. As $2R > r$,  we have that $g'$ is $(\epsilon, R)$-deep in some
$Y \in \mathbb Y$. Since $d(g, Y) < r + \epsilon$, we have $Y
\in \mathbb F$.

Let $z$ be the exit point of $\gamma$ in $N_\epsilon(Y)$, which exists since $\xi$ is conical. Since $d(x, z)\ge R$, by
Lemma \ref{entrytrans}, we have that $z$ is an $(\epsilon, R)$-transition point
in $\gamma$.  Hence, $d(z, g') >2R-r$ and then $d(z, g)\ge 2R-2r \ge R_0$.

Noting that $z\in [g', \xi]_\gamma$, we have $d(1, z) +2r > d(1, g) + d(g, z)$. Taking into account (\ref{SUM}), we estimate
$\mu_1(\Xi)$ as follows:
\begin{equation}\label{UBND}
\begin{array}{rl}
\mu_1(\Xi) & \le \sum\limits_{Y \in \mathbb F} \sum\limits_{z \in N_\epsilon(Y)}^{d(z, g) > R_0} \mu_1(\Pi_r(z))  \le C_2\cdot \sum\limits_{Y \in \mathbb F} \sum\limits_{z \in N_\epsilon(Y)}^{d(z, g) > R_0} \exp(-\g G d(1, z)) \\
&\le C_2 \cdot \exp(-\g G d(1, g)) \cdot \sum\limits_{Y \in \mathbb F} \sum\limits_{z \in N_\epsilon(Y)}^{d(z, g) > R_0}
 \exp(-\g G d(g, z)) \cdot \exp(2\g G r)\\
& \le \exp(-\g G d(1, g)) \cdot C_1/2.
\end{array}
\end{equation}
By Lemma \ref{parabnoatom}, we have $\mu_1(\Pi_{r}^c(g))=\mu_1(\Pi_{r}(g))$ and so   $$\mu_1(\Xi) + \mu_1(\Pi_{r, \epsilon, R}(g)) =
\mu_1(\Pi_r(g)).$$ So the inequalities
(\ref{LBND}) and (\ref{UBND}) yield $$\mu_1(\Pi_{r, \epsilon,
R}(g)) \ge (C_1/2) \exp(-\g G d(1, g)).$$ The proof is now complete.
\end{proof}

\subsection{Growth of partial cones}\label{SSpartialconegrowth}
We are in a position to prove the exponential growth of (the annulus
of) a partial cone. 

\begin{prop}\label{PConeGrowthG}
Let $r, \epsilon, R$ be as in Lemma \ref{PShadowLem}. There are
constants $\Delta, \kappa>0$ such that the following holds
$$
\sharp {\Omega_{r, \epsilon, R}(g, n, \Delta)} \ge \kappa \cdot
\exp(n\g G)
$$
for any $g \in G$ and $n \ge 3R+r$.
\end{prop}
\begin{proof}
Fix any $\Delta\ge 0$. We first note the following observation, roughly saying that the
partial shadow $\Pi_{r, \epsilon, R}(g)$ can be covered by those of
all elements at level $n$ in $\Omega_{r, \epsilon, R}(g, n, \Delta)$.

\begin{claim}
The
following holds
$$
\mu_1(\Pi_{r, \epsilon, R}(g)) \le \sum\limits_{h \in \Omega_{r,
\epsilon, R}(g,n, \Delta)} \mu_1(\Pi_{r}(h)),
$$
for any $n \ge 3R+r$.
\end{claim}
\begin{proof}[Proof of Claim]
Let $\xi \in \Pi_{r, \epsilon, R}(g)$. By definition, there exists a
geodesic $\gamma=[1, \xi]$ such that $\gamma$ contains an $(\epsilon, R)$-transition point $v$ in $B(g, 2R)$ and $\gamma \cap B(g,
r) \neq \emptyset$.

Let $h \in \gamma$ such that $d(h, 1)=n+d(1, g)$. Since $n
\ge 3R+r$, it follows that $d(g, [1, h]_\gamma)< r$ and $B(v, R) \cap \gamma  \subset
[1,h]_\gamma$. Then $v$ is an $(\epsilon, R)$-transition point in
$[1, h]_\gamma$. Thus, $h \in \Omega_{r,
\epsilon, R}(g, n, 0)$. Hence, the claim is proved.
\end{proof}

By Shadow Lemmas \ref{ShadowLem} and \ref{PShadowLem},  it follows
that
$$
\exp(n\g G) \cdot \mu_1(\Pi_r(h))  \asymp_{\epsilon, r, R,
\Delta} \mu_1(\Pi_{r,\epsilon, R}(g)),
$$
for any $h \in \Omega_{r, \epsilon, R}(g,n, \Delta)$. By the Claim
above we obtain that $$\sharp \Omega_{r, \epsilon, R+D}(g,n, \Delta)
\succ_{r, \epsilon, R, \Delta} \exp(n\g G)$$ for $n\ge 3R+r$. The
proof is complete.
\end{proof}

The \textit{dead-end depth} of an element $g$ in $\Gx$ is a non-negative integer strictly less than the length of a shortest word $w \in G$ such that $d(1, gw) > d(1, g)$. If the dead-end depth is non-zero, then $g$ is called a \textit{dead end} (i.e. the geodesic $[1, g]$ could not be extended furthermore). Such elements exist, for instance, in the Cayley
graph of $G \times Z_2$ with respect to a particular generating set. See \cite{GriH} for a brief discussion and
references therein. However, one could get around the dead end by the following result,
saying that for any $g \in G$, there exists a companion element near
$g$ with a very large cone.
\begin{lem}[Companion Cone]\label{CompanionCone}
Let $r, \epsilon, R$ be as in Lemma \ref{PShadowLem}. There are
constants $\Delta, \kappa>0$ with the following property.

For any $g \in G$, there exists $g' \in G$ such that $d(g, g') < r$
and the following holds $$ \sharp \Omega_{\epsilon, R}(g', n, \Delta)
\ge \kappa \cdot \exp(n\g G)
$$
for any $n \gg 0$. Moreover, we can choose $g'=1$ if $g=1$.
\end{lem}
\begin{proof}
Let $r, \epsilon, R, \Delta, \kappa$ be given by Proposition \ref{PConeGrowthG} such that $$
\sharp {\Omega_{r, \epsilon, R}(g, n, \Delta)} \ge \kappa \cdot
\exp(n\g G)
$$
for any $g \in G$ and $n \gg 0$. By definition, there exists a geodesic $\omega_h=[1, h]$ for $h \in \Omega_{r, \epsilon, R}(g, n, \Delta)$ such that $\omega_h$ contains an $(\epsilon, R)$-transition point in $B(g, 2R)$.  We apply a pigeonhole principle to the finite set $B(g, r)$: there exists $g'\in B(g, r)$ such that $$\sharp \Omega_{\epsilon, R+r}(g', n, \Delta+r) \ge  \frac{\kappa}{\sharp B(g, r)}  \cdot
\exp(n\g G).$$ If $g = 1$, this holds for $g'=1$ without using the pigeonhole principle. 
\end{proof}

\subsection{Large Transitional Trees}
In this subsection, we shall construct a tree $\mathcal T$ in $\Gx$
such that the growth rate of $\mathcal T$ is sufficiently close to
$\g G$ and transition points are uniformly spaced in $\mathcal T$.
Our construction is inspired by a nice argument of Coulon \cite{Coulon} in a
hyperbolic group.

The following lemma is useful in choosing a sufficiently separated
set. Recall that a metric space $(A, d)$ is \textit{$C$-separated} for $C>0$ if $d(a,a')
>C$ for any two distinct $a,a' \in A$.

\begin{lem}\label{sepnet}\cite[Lemma 5]{AL}
For any $C>0$ there exists a constant $\theta=\theta(C)
>0$ depending on $G$ and $S$ with the following property.

Let $Y$ be any finite set in $G$. Then there exists a $C$-separated
subset $Z \subset Y$ such that $\sharp Z > \theta \cdot \sharp Y$.
\end{lem}

Take $r, \epsilon, R, \Delta, \kappa>0$ which are given by Lemma
\ref{CompanionCone} such that the auxiliary set
\begin{equation}\label{essentialset}
E = \{g \in G: \sharp \Omega_{\epsilon, R}(g, n, \Delta) >
\kappa \cdot \exp(n\g G),  \forall n \ge 0\},
\end{equation}
contains $1$ and $N_r(E) = G$.

Fix $n,\; \kappa',\; C>0$. We make use of a
construction of Coulon in \cite{Coulon}, defining a sequence of
decreasing sets $G_{i,n}$ for $i \ge 0$ inductively. Let $G_{0,n} =
G$. Then set
$$G_{i+1,n} = \{g \in G_{i,n}: \sharp(\Omega_{\epsilon, 2C}(g, n, \Delta+3C) \cap G_{i,n}) > \kappa' \cdot \exp(n\g G)\}.$$

\begin{lem}\label{PreTreeSetG}
There exist $\kappa', C>0$ such that the following holds $$B(g, C) \cap G_{i, n} \neq
\emptyset$$
for each $g \in E$ and $i \ge 0,\; n\gg 0$. Moreover, $1 \in G_{i, n}$ for $i\ge 0$.
\end{lem}
\begin{proof}
The constant $C$ is given below in (\ref{ValueC}). We proceed by induction on $i$. For $i=0$, it is automatically true.
Assume that $B(g, C) \cap G_{j, n} \neq \emptyset$ and $1\in G_{i, n}$ for each $j \le
i$, $g \in E$. We shall show that $$B(g, C) \cap G_{i+1,n} \neq
\emptyset$$ for each $g \in E, n\gg 0$.

By definition of the set $E$ in (\ref{essentialset}), $\sharp
\Omega_{\epsilon, R}(g, n, \Delta) \ge \kappa \cdot \exp(\g G
n)$ for any $g \in E$. By Lemma \ref{sepnet}, there exist
$\theta=\theta(4C)>0$ and a $4C$-separated set $Z \subset
\Omega_{\epsilon, R}(g, n, \Delta)$ such that
\begin{equation}\label{Zset}
\sharp Z > \kappa \theta \cdot \exp(\g G n).
\end{equation}

Let $x \in Z$. By the inductive assumption on $i$ and $N_r(E)=G$,
there exists $h_x \in G_{i,n}$ such that $d(x, h_x) < 2C$.

Observe that $h_x \neq h_y$ for any $x \neq y \in Z$. Indeed, if
$h_x = h_y$, then $d(x, y) < 4C$ for $x, y \in Z$. This
contradicts to the choice of $Z$ as a $4C$-separated set in
$\Omega_{\epsilon, R}(g, n, \Delta)$. By (\ref{Zset}), the following holds 
\begin{equation}\label{Zset2}
\sharp \{h_x: x \in Z\} > \kappa \theta \cdot \exp(\g G n).
\end{equation}

Set $\displaystyle{\kappa':=\frac{\kappa\theta}{\sharp B(1, C)}}$. We next
show the following claim.
\begin{claim}
There exists an element $g' \in B(g, C)$ such that
$$\sharp (\Omega_{\epsilon, 2C}(g', n, \Delta+3C) \cap G_{i,n}) > \kappa'
\cdot \exp(n\g G).$$
Moreover, if $g=1$, we can choose $g'=1$.
\end{claim}
\begin{proof}[Proof of Claim]
Since $x \in Z \subset \Omega_{\epsilon, R}(g, n, \Delta)$, a geodesic $\gamma=[1,x]$ contains an $(\epsilon, R)$-transition
point $v \in \gamma \cap B(g, 2R)$. Note that $d(g, x) > n-\Delta$ and $d(g, v) \le 2R$.  Let $D=D(\epsilon, R), L=L(2C, \epsilon, R)>0$ be given by Lemma
\ref{PairTransG}, where 
\begin{equation}\label{ValueC}
C:=\max\{r, D+2R\}.
\end{equation}
By choosing $n \ge L +2R+\Delta$ we obtain
$d(v, x) > L$.

We apply Lemma
\ref{PairTransG} to the pair of geodesics $\gamma, \omega_x:=[1, h_x]$, where $d(x, h_x) < 2C$.
Then there exists an $(\epsilon, R)$-transition point $w$ in
$\omega_x$ such that $d(v, w) < D$. Hence, we have for any $h_x$ ($x  \in X$):
$$d(g, \omega_x) < D+2R \le C,$$ where $C$ is given by (\ref{ValueC}).

Since $d(x, h_x) \le 2C$ and $x \in \Omega_{\epsilon, R}(g, n, \Delta)$, we have $|d(h_x, g')- n|\le\Delta+3C$ for any $g' \in B(g, C)$. By using a pigeonhole principle on
$B(g, C)$, there exist $g' \in B(g, C)$ and a subset $$Y \subset
\{h_x: x \in Z\} \cap  \Omega(g', n, \Delta+3C)$$ with the following two properties:
\begin{itemize}
\item
the
following holds
\begin{equation}\label{YLBD}
\sharp (Y \cap G_{i,n}) > \kappa\theta /\sharp B(1, C) \cdot
\exp(n \g G) \ge \kappa' \cdot \exp(n \g G);
\end{equation}
\item for each $h_x \in Y$, the geodesic $\beta$ contains an
$(\epsilon, R)$-transition point $w$ such that $d(w, g') < 2C$.
\end{itemize}
Hence, $Y \subset \Omega_{\epsilon,
2C}(g', n, \Delta+3C)$ and (\ref{YLBD}) yields
$$\sharp (\Omega_{\epsilon, 2C}(g', n, \Delta+3C) \cap G_{i,n}) > \kappa'
\cdot \exp(n\g G).$$ If $g=1$, we can let $g'=1$ as in Lemma \ref{CompanionCone}. The claim is thus proved. 
\end{proof}

By the
Claim above, we have for any $j \le i$:
$$\sharp (\Omega_{\epsilon, 2C}(g', n, \Delta+3C) \cap G_{j,n})   > \kappa'
\cdot \exp(n \g G)$$  since $G_{i,n} \subset G_{j,n}$ . By definition of $G_i$, we see $g' \in
G_{i,n}$ and thus $g' \in G_{i+1, n}$ by the Claim again. Therefore, $B(g,
C) \cap G_{i+1, n} \neq \emptyset$ and $1\in G_{i+1, n}$.
\end{proof}

\begin{lem}\label{TreeSetG}
There exist $\epsilon, R, \Delta, \kappa>0$ with the following
property.

For any $n
\gg 0$ there exists a subset $1\in \hat G$ in $G$ with the following
property
$$\sharp(\Omega_{\epsilon, R}(g, n,\Delta) \cap \hat G) > \kappa \cdot\exp(n\g G)$$
for any $g \in \hat G$.
\end{lem}
\begin{proof}
Denote $R=2C, \Delta=\Delta+3C, \kappa=\kappa'$, where $\kappa'$ is given as above. Defining $\hat G = \bigcap_{i\ge 0}
G_{i, n}$ completes the proof. 
\end{proof}

\begin{thm}[Large Trees]\label{LargeTreeG}
Let $\epsilon, R, \Delta$ be given by Lemma \ref{TreeSetG}, satisfying also
Convention \ref{epsilonR}.

For any $0 <\sigma < \g G$, there exist $L >0$ and a rooted geodesic tree
$\mathcal T$ at $1$ in $\Gx$ with the following properties.
\begin{enumerate}
\item
Let $\gamma$ be a geodesic in $\mathcal T$ such that $\gamma_-=1$. For any $x \in \gamma$, there exists an $(\epsilon, R)$-transition point $v$ in $\gamma$
such that $d(x, v)<L$.
\item
Let $\g{\mathcal T}$ be the growth rate of $\mathcal T$ in $G$. Then
$\g{\mathcal T} \ge \sigma$.
\end{enumerate}
\end{thm}
\begin{proof}
Let $1 \in \hat G$ be given by Lemma \ref{TreeSetG}. Without loss of generality assume that $\kappa < 1$. Fix $L>0$. The root of
$\mathcal T$ is $\mathcal T_0=\{1\}$. Assume that $\mathcal T_i$ is
defined for $i\ge 0$. For each $x \in \mathcal T_i$, choose a subset $\mathcal T_i(x)$ in $\Omega_{\epsilon, R}(x, L,\Delta) \cap \hat G$ such that
$$
\sharp \mathcal T_i(x) \ge \kappa \cdot\exp(L\g G),
$$
and $\mathcal T_i(x)$ is $C$-separated, where $C$ is defined below.
Then set
\begin{equation}\label{GeodTree}
\mathcal T_{i+1} = \cup_{x \in \mathcal T_i} (\cup_{y \in \mathcal
T_i(x)}[1, x] \cdot [x, y]),
\end{equation}
where the geodesic $[1, x]$ is the one in $\mathcal T_i$ and $[x,y]$ is any choice of geodesic. 

Observe that $\mathcal T_k$ is a rooted tree at $1$ for each $k>0$. Suppose not, let $\alpha, \beta$ be a pair of geodesics in $\mathcal T_k$ with the common endpoints $z \in \mathcal T_i$ and $w\in \mathcal T_j$ ($i, j \le k$) such that $j-i>0$ is minimal. Let $x=\alpha\cap \mathcal T_i(z)$ and $y=\beta\cap \mathcal T_i(z)$ where $T_i(z) \subset  \Omega_{\epsilon, R}(z, L,\Delta)$. So we have  $x\ne y$ and $j-i\ge 2$. Since $x$ has a large partial cone by construction, there exists an $(\epsilon, R)$-transition point $\tilde x$ in $\alpha$ such that $d(x, \tilde x)\le 2R$.   By Lemmas \ref{floydtrans} and \ref{karlssonlem} there exists $D=D(\epsilon, R)>0$ such that $d(\tilde x, \beta)\le D$ and so $d(x, \beta) \le D+2R$. Since $x, y\in \Omega_{\epsilon, R}(z, L,\Delta)$ we have $|d(x, z)-d(y, z)|\le 2\Delta$. By a standard argument, we have $d(x, y)\le 2(D+2R+\Delta)$. Define $$C=2(D+2R+\Delta)+1.$$ This gives a contradiction as $\mathcal T_i(z)$ is $C$-separated. Thus $\mathcal T_i$ is a rooted tree.  

We take the union $\mathcal T =\cup_{i\to
\infty} \mathcal T_i$. By construction, for any point $x \in \gamma$ there exists an $(\epsilon, R)$-transition point $v$ such that $d(x, v) \le L+2R+\Delta$. The first statement is true. For the statement (2), we compute
$$\sharp B(1, n(L+\Delta)) \cap \mathcal T >\kappa^n \cdot\exp(nL\cdot \g
G),$$ which yields $\g{\mathcal T} >  (\log \kappa + r \g
G)/(L+\Delta)$. Hence, $\g{\mathcal T} \to \g G$ as $L \to \infty$. This concludes the proof.
\end{proof}


\section{Lift paths in nerve graphs}\label{Section6}
Recall that $G,\; X,\; \mathbb U,\;  \mathcal U$
are as in Definition \ref{RHdefn}. Fix a finite generating set $S$. 

The aim of this section
is to make precise a relation between the graph $\Gx$ and the
orbit structure for the action of $G$ on $X$. Most results here are well-known, cf \cite[Section 6]{Bow1}, but the presentation is adapted to our purpose.


Fix a basepoint $o \in X\setminus \mathcal U$. Without loss of generality, assume that any geodesic $[o, so]$ for $s\in S$ does not intersect $\mathcal U$ by shrinking equivariantly $\mathbb U$. Choose a
projection point $o_U$ of $o$ to each $U \in \mathbb{ U}$. We choose, and denote by $\mathbb{\bar{U}}$, a complete set of \textit{shortest} $G$-representatives
$U$ in $\mathbb U$ such that $d(o, U)=d(o, gU)$ for any $g\in G$. The
\textit{nerve graph} $\Gr$ over $Go \cup \mathbb U$ is constructed as
follows.
\subsection{Nerve graph}
Let $V(\Gr) = Go \cup \mathbb U$ be the union of two types of
vertices: \textit{orbit vertices} $Go$ and \textit{horoball
vertices} $\mathbb U$. Connect $o$ to $o_U$ by edge $(o, o_U)$ for
each $U \in \mathbb{\bar{U}}$, and $o$ to $so$ by edge $(o, so)$ for
each $s \in S$. This incident relation extends over $V(\Gr)$ in a
$G$-equivariant manner. There are no edges between two horoball
vertices.   It is obvious that
$\Gr$ is connected, on which $G$ acts co-compactly but not properly
in general.

By construction, we can set the length of (translated) edges $(o,
o_U)$(resp. $(o, so)$) to be $d(o, o_U)$(resp. $d(o, so)$). This
gives a geodesic length metric $d_\Gr$ on $\Gr$.

Define
\begin{equation}\label{M}
M=\max\{d(o,o_U), d(o, so): U\in \mathbb{\bar{U}}; s\in S\},
\end{equation}
for which $X \setminus \mathcal U \subset N_M(Go)$.

Note that  $(\Gr, d_\Gr)$ is quasi-isometric to the cone-off Cayley graph $\G$ in the sense of Farb \cite{Farb}.  The next 
result thus follows from \cite[Proposition 8.13]{Hru}, plus the
fact that $G$ acts properly on $X$.

\begin{lem}\label{relgeodesic}
There exists a constant $D>0$ such that the following holds. Assume
that  $\alpha=[1, g]$ a geodesic in $\Gx$, and $\gamma$ is a geodesic in $(\Gr, d_\Gr)$ between $o$ and
$go$. For any orbit vertex $ho
\in \gamma$, there exists $f \in \alpha$ such that $d(ho, fo) < D$.
\end{lem}

\subsection{Lift paths}
Given a path $\gamma$ in $\mathcal G$, we shall construct a lift path
$\hat \gamma$ in $X$. Fix a (choice of) geodesic $[o, o_U]$ for each $U
\in \mathbb{\bar U}$, and a geodesic $[o, so]$ for each $s \in S$.

If $\gamma$ contains no horoball vertices, then $p$ can be
naturally seen as a concatenation path in $X$ consisting of
(translated) geodesic segments $[o, so]$ for edges labeled by $s$ in
$\gamma$. By the choice of $ \mathbb{U}$, the path $\gamma$ lies entirely in $X\setminus \mathcal U$.

In order to obtain the lift path in general case, it suffices to
modify pairs of edges adjacent to a horoball vertex.

\begin{defnconst}[Lift path in $\mathcal G$] \label{defnlift}
Let $\gamma$ be a path in $\mathcal G$ with at least one orbit vertex.
The \textit{lift} of $\gamma$ is a path $\hat \gamma$ in $X$ obtained as
follows.

Let $U \in \mathbb U$ be a horoball vertex in $\gamma$. Denote by $U_-,
U_+$ respectively, if exist, the previous and next vertices adjacent to $U$
in $\gamma$. Choose projection points $u_-, u_+$, if exist, of $U_-, U_+$
on the horoball $U$ respectively.

If $U =\gamma_-$(resp. $U =\gamma_+$), we replace the edge $(U, U_+)$(resp.
$(U_-, U)$) of $\gamma$ by a geodesic $[u_-, U_+]$(resp. $[U_-, u_+]$) in
$X$.

If $U\neq \gamma_-, \gamma_+$, then $U_-, U_+ \in Go$. We replace
the subpath $[U_-, U_+]_\gamma$ of $p$ by the concatenated path
$$[U_-, u_-][u_-, u_+][u_+, U_+]$$ in $X$.

Repeating the above procedure for every horoball vertex $U$, we get the
lift $\hat \gamma$.
\end{defnconst}
\begin{rem}\label{decomplift}
With notations in definition \ref{defnlift}, the lift path $\hat \gamma$
can be divided into the following form:
\begin{equation}\label{liftpathform}
\hat \gamma = \hat \gamma_1 \cdot [(u_1)_-, (u_1)_+] \cdot \hat \gamma_2 \dots \hat \gamma_i \cdot [(u_i)_-, (u_i)_+] \cdot \hat \gamma_{i+1} \dots,
\end{equation}
where $\gamma_i$ are maximal subsegments of $\gamma$ with
horoball vertices at endpoints, and $U$ are the common horoball
vertex with $U=(\gamma_i)_+=(\gamma_{i+1})_-$.
\end{rem}

\begin{lem}\label{lift}
The lift of a quasi-geodesic in $(\Gr, d_\Gr)$ is a quasi-geodesic
in $(X, d)$.
\end{lem}
\begin{proof}
The path $\hat p$ is an efficient semi-polygonal path in the sense of Bowditch in \cite[Section 7]{Bow1}. The lemma follows as a consequence of Lemma 7.3 in \cite{Bow1}.
\end{proof}

\subsection{Transitions points revisited}
In this subsection, we discuss the notion of transition points in
$X$ with respect to the horoball system $\mathbb U$. Recall that $\delta>0$ is the hyperbolicity constant of $X$.

The following lemma says that the image of a geodesic in $\Gx$
follows closely transition points of any geodesic with the same
endpoints in $X$.
\begin{lem}\label{TransCloseG}
There exist $\epsilon>0$ and $D = D(\epsilon, R)$ for any $R
>0$ such that the following holds.

Given $g \in G$, consider  a
word geodesic $\alpha=[1, g]$ in $\Gx$ and a geodesic $\beta=[o, go]$ in $X$. Then for any $(\epsilon,
R)$-transition point $x \in \beta$, there exists a vertex $h \in
\alpha$ such that $d(ho, x)< D$.
\end{lem}
\begin{proof}
Let $\gamma$ be a geodesic with endpoints $o, go$ in $(\Gr, d_\Gr)$. By Lemma \ref{lift},  the lift path $\hat \gamma$ is a $(\lambda,
c)$-quasi-geodesic in $X$ for some $\lambda, c>0$. By the stability of quasi-geodesics, there exists $\varepsilon=\varepsilon(\lambda, c)>0$ such that $\beta$ and $\hat \gamma$ have a Hasudorff distance at most $\varepsilon$. Write the lift path $\hat \gamma$ as in (\ref{liftpathform}) and so we have two cases to consider: 

\textbf{Case 1}: $x$ has a $\varepsilon$-distance to $\hat \gamma_i$ for some $i$. For $M$ given in (\ref{M}), $x$ has a distance at most $(\varepsilon+M)$  to an orbit vertex $v$ of $\gamma$. Let $D_1$ be given by Lemma \ref{relgeodesic} so that there exists $h \in \alpha$
such that $d(ho, v) < D_1$ and so $d(x, ho)\le M+D_1+\varepsilon$.   

\textbf{Case 2}: $d(x, [u_-, u_+])\le \varepsilon$, where $u_-, u_+ \in U$ for some $U$.  By the stability of quasi-geodesics, we have $d(u_-, y), d(u_+, z)\le \varepsilon$ for some $y, z\in \beta$. Without loss of generality, assume that $x\in [y, z]_\beta$.

Let $L=L(R)$ given by Lemma \ref{uniformtrans} such that if the following holds $$\min\{d(x, y)-d(y, U),\; d(x, z)-d(z, U)\}\ge L,$$ then $x$ is $(\epsilon, R)$-deep in $U$.  So,  $x$ has a distance at most $L+\varepsilon$ to one of $u_-, u_+$, both of which have a distance at most $M$ (\ref{M}) to an orbit vertex in $p$. Then as the Case (1), there exists $h\in \alpha$ such that $d(ho, x) \le D_1+M+L+\varepsilon$.  The lemma is proved.
\end{proof}
\begin{rem}
Since $G$ acts co-compactly on $X \setminus \mathcal U$, we can
always find $h \in G$ such that $d(ho, x) < D$. The point of
the lemma lies in the statement that $h$ can be chosen to be lying on
the geodesic $[1, g]$.
\end{rem}

\begin{conv}[about $\epsilon, R$]\label{metricepsilonR}
When talking about $(\epsilon, R)$-transition points in $X$ we
always assume that $\epsilon>\epsilon_0$, where $\epsilon_0$ are
given by Lemmas \ref{uniformtrans}, \ref{entrytrans}, and
\ref{TransCloseG}. In addition, assume that $R >R(\epsilon)$, where
$R(\epsilon)$ is given by Lemma \ref{entrytrans}.
\end{conv}

The following, an analogue of Lemma \ref{PairTransG}, will play the same role in cusp-uniform actions. 
\begin{lem}\cite[Lemma 2.13]{YANG8}\label{PairTransX}
Let $\epsilon, R>0$ be in Convention (\ref{metricepsilonR}). For any
$r>0$, there exist  $D=D(\epsilon, R), L =L(\epsilon, R, r)>0$ with
the following property.

Let $\alpha, \gamma$ be two geodesics in $X$ such that
$\alpha_-=\gamma_-, \; d(\alpha_+, \gamma_+) < r$. Take an
$(\epsilon, R)$-transition point $v$ in $\alpha$ such that $d(v,
\alpha_+) >L$. Then there exists an $(\epsilon, R)$-transition point
$w$ in $\gamma$ such that $d(v, w) < D$.
\end{lem}

\section{Large quasi-trees for a cusp-uniform action}\label{Section7}
Recall that $\GP$ is relatively hyperbolic, and $X,\;  \mathcal U,\;
\mathbb U,\; \mathbb{\bar U}$ are as in Definition
\ref{RHdefn}. Suppose that $G$ has the parabolic gap property.

\subsection{Growth of the orbit in partial cones}
We consider a contracting system $\mathbb Y$ in $X$ as in Convention \ref{ConvContracting} such that $\mathbb U\subset \mathbb Y$. The partial cone $\Omega_{r, \epsilon, R}(go)$ is defined using transition points relative to $\mathbb Y$. The following proposition holds by a more general result in \cite[Corollary 1.9]{YANG8}.

\begin{prop}\label{PConeGrowthX}
There are
constants $r, \epsilon, R, \Delta, \kappa
>0$ such that the following holds
$$
\sharp \Omega_{r, \epsilon, R}(go, n, \Delta) \ge \kappa \cdot
\exp(n\g G)
$$
for any $g\in G, n \ge 0$.
\end{prop}
\begin{rem}
In \cite{YANG8} we characterize the purely exponential growth of (partial) cones by a condition introduced by Dal'bo-Otal-Peign\'e \cite{DOP}. The parabolic gap property ensures their condition to hold.
\end{rem}

\subsection{Large Quasi-Trees in Hyperbolic Spaces}
This subsection is devoted to an analogue of Theorem
\ref{LargeTreeG}. The difficulty to do so is that there is no
obvious way to ``label" geodesics with endpoints in $Go$ in the space
$X$. Our strategy is to map into $X$ the set of
geodesics between $1$ and $g$ in $\Gx$ so that the existence of the transition points
in the partial cone will force these (images of) geodesics
travel uniformly close to the base $go$ of the cone. This allows us
to adapt the methods in proving Theorem \ref{LargeTreeG} in this
setting.

Recall that $\Omega(g)$ denotes the cone defined in the Cayley graph $\Gx$ (See Definition \ref{defnpcone}). Define 
$$\mathcal L_{r, \epsilon, R}(go, n,
\Delta) := \Omega(g) \cap \Omega_{r, \epsilon, R}(go, n, \Delta).$$ Said differently, $\mathcal L_{r, \epsilon, R}(go, n,
\Delta)$ is a subset of elements $h \in \Omega_{r, \epsilon, R}(go, n, \Delta)$ such that some geodesic   $[1, h]$ in $\Gx$ passes through $g$.

\begin{lem}[Companion cone]\label{metriccompanion}
There are constants $r, \epsilon, R, \Delta, \kappa>0$ such that the
following holds.

For any $g \in G$, there exists $g' \in G$ such that $d(go, g'o) <
r$ and the following holds $$ \sharp  \mathcal L_{r, \epsilon, R}(g'o,
n, \Delta) \ge \kappa \cdot \exp(n\g G)
$$
for any $n \ge 0$. If $g=1$, then $g'=1$.
\end{lem}
\begin{proof}
The proof proceeds similarly as that of Lemma \ref{CompanionCone},  making use of Lemma
\ref{TransCloseG} instead of Lemma
\ref{PairTransG}. We leave the proof to the interested reader. 
\end{proof}
 
The proof of Lemma \ref{PreTreeSet} below goes similarly as that of Lemma
\ref{PreTreeSetG}. We single out the differences only and refer the reader to the proof Lemma \ref{PreTreeSetG}.

Take $r, \epsilon, R, \Delta, \kappa>0$ by Lemma
\ref{metriccompanion}. Define the auxiliary set
\begin{equation}\label{companionset}
E = \{g \in G: \sharp \mathcal L_{r, \epsilon, R}(go, n, \Delta) >
\kappa \cdot \exp(n\g G),  \forall n \ge 0\}.
\end{equation} Note that $N_r(Eo) = Go$.

Let $D=D(\epsilon, R)$ satisfy Lemmas \ref{TransCloseG},
\ref{PairTransX}, and $L=L(2C)>0$ be given by Lemma \ref{PairTransX},
where $$C:=D+2R.$$

By Lemma \ref{sepnet} and the proper action of $G$ on $X$, there
exists $\theta=\theta(4C)>0$ with the following property. For any
finite subset $Y$ in $G$, there exists a subset $Z$ in $Y$ such that
$Zo$ is $(4C)$-separated in $X$ and $\sharp Z
>\theta \cdot \sharp Y$.

Let $\kappa' =\kappa\theta/\sharp N(o, C)$. Fix $n>0$. We define a sequence of
decreasing sets $G_{i,n}$ for $i \ge 0$ inductively. Let $G_{0,n} =
G$. Then set
$$G_{i+1,n} = \{g \in G_{i,n}: \sharp(\mathcal L_{r, \epsilon, 2C}(go, n, \Delta+3C) \cap G_{i,n}) > \kappa' \cdot \exp(n\g G)\}.$$

\begin{lem}\label{PreTreeSet}
For $n\gg 0$, the following holds $$B(go, C) \cap G_{i, n}o \neq
\emptyset$$
 for each $g \in E$ and $i \ge 0$. Moreover, $1 \in G_{i, n}$ for $i\ge 0$.
\end{lem}
\begin{proof}
The argument is completely analogous, except that Lemma
\ref{TransCloseG} is used instead of Lemma \ref{PairTransG}.
\end{proof}

The analogue of Lemma \ref{PreTreeSetG} therefore follows.

\begin{lem}\label{TreeSetX}
There exist $r, \epsilon, R, \Delta,\kappa>0$ such that for any $n
\gg 0$, there exists a subset $1\in \hat G$ in $G$ with the following
property
$$\sharp( \mathcal L_{r, \epsilon, R}(go, n,\Delta) \cap \hat G) > \kappa \cdot\exp(n\g G)$$
for any $g \in \hat G$.
\end{lem}

The following is analogous to Theorem \ref{LargeTreeG}.
\begin{thm}[Large Quasi-Trees]\label{LargeTreeX}
There exist $\epsilon, R, D>0$ such that the following holds. For
any $0 <\sigma < \g G$, there exist $L >0$ and a rooted geodesic tree
$\mathcal T$ at $1$ in $\Gx$ with the following properties.
\begin{enumerate}
\item
For any vertex $g$ in $\mathcal T$, the image of a geodesic
$[1, g]$ has at most a $D$-Hausdorff distance to any geodesic $[o,
go]$.
\item
Consider a geodesic $\gamma= [o,go]$. For any $x \in \gamma$, there
exists an $(\epsilon, R)$-transition point $y \in \gamma$ such that
$d(x, y) < L$.
\item
Let $\g{\mathcal T}$ be the growth rate of the set $\mathcal To$ in  
$X$. Then $\g{\mathcal T} \ge \sigma$.
\end{enumerate}
\end{thm}

\begin{proof}
Let $r, \epsilon, R, \Delta,\kappa>0$ be given by Lemma
\ref{TreeSetX}. For any $L>0$, the geodesic tree
$\mathcal T=\cup_{i\to \infty} \mathcal T_i$ is constructed in $\Gx$ exactly
as in the proof of Theorem \ref{LargeTreeG}, cf. (\ref{GeodTree}). By
construction,  $|d(xo, yo) -L|< \Delta$ for any $y \in \mathcal
T'_{i}(x)$. Following $\mathcal T$, we construct a quasi-embedded
tree in $X$. Let $\mathcal T'_0=\{o\}$ and $$\mathcal T'_{i+1} = \cup_{x
\in \mathcal T_i} (\cup_{y \in \mathcal T_{i+1}(x)} [xo, yo]),$$
where $[xo, yo]$ is a choice of a geodesic between $xo, yo$. Set
$\mathcal T': = \cup_{i\ge 0} \mathcal T'_i$. This is indeed a quasi-tree by the following claim.  
\begin{claim}
Assume $L \gg 0$. Then there exist $\lambda=\lambda(r), c=c(r)$
such that each path $p$ originating at $o$ in $\mathcal T'$ is a $(\lambda,
c)$-quasi-geodesic.
\end{claim}
\begin{proof}[Proof of Claim]
We shall show that any subpath of length at most $2(L-\Delta)$ is a $(1, 2r+2\delta)$-quasi-geodesic. Assume that $p = [o, xo][xo, yo][yo, zo]\cdots$, where $x \in
\mathcal T_0(1), y \in \mathcal T_1(x), z \in \mathcal T_2(y)$. As
$y \in \Omega_{r, \epsilon, R}(xo, L, \Delta)$, it follows that $d(xo, [o,
yo]) < r$. Hence, $[o,yo]_p$ is a $(1, 2r+2\delta)$-quasi-geodesic.

Let $w \in [o, zo]$ such that
$d(w, yo) < r$. Observe that $d(w, [xo, zo]) \le \delta$ for $L\gg 0$. If not,
by the $\delta$-thin triangle property there exists $w' \in [o,xo]$ such that $d(w,
w') \le \delta$. As $[o, yo]_p$ is $(1,2r+2\delta)$-quasi-geodesic, we have
$d(xo, yo) \le \len([w', yo]_p) \le d(w', yo) +2r+2\delta< 2r+3\delta$. As $\mathcal
T_i(x) \subset \Omega_{r, \epsilon, R}(xo, L, \Delta)$, we have
$d(xo, yo) >L-\Delta$. Assuming $L > 2r+\Delta+3\delta$ would give a
contradiction.  Hence, $d(w, [xo, zo]) \le \delta$ and 
$d(yo, [xo, zo])\le r + \delta$. As a consequence, $[xo, zo]_p$ is
a $(1, 2r+2\delta)$-quasi-geodesic.

Inductively, we prove that $p$ is a local quasi-geodesic as stated above. By
taking $L$ large enough, we see that $p$ becomes a global
quasi-geodesic.
\end{proof}

Once we proved that $\mathcal T'$ is a quasi-tree, the estimate of the growth
rate  of $\mathcal To$ is the same as in the proof of Theorem
\ref{LargeTreeG}. So for any $\delta <\g G$ we can find $L>0$ such that
the statement (3) holds.

We now prove the statement (2). Let $p$ be the path in $\mathcal T'$ between $o$
and $go$. By the above Claim, there exists $D_1=D(\lambda, c)$ such
that for any $x\in \gamma$, we have $ho \in p$ for some $h \in \mathcal T$
with $d(x, ho) < D_1$. For concreteness, assume $h \in \mathcal L_{r, \epsilon,
R}(zo, L, \Delta)$ for some $z \in \mathcal T$. Thus, some geodesic $[o,
ho]$ contains an $(\epsilon, R)$-transition point $w$ in $B(zo,
2R)$. 

Let $D_2=D(\epsilon, R), L_1=L(\epsilon, R, D_1)$ be given by Lemma
\ref{PairTransX}. Note that $d(x, ho)\le D_1$ and $d(w, ho)\ge d(ho, zo)-d(zo, w) \ge L-2R$. We assume that $L>L_1+2R$ and then apply Lemma
\ref{PairTransX} to the pair of geodesics $[o, x]_\gamma$ and $[o, ho]$: there exists an
$(\epsilon, R)$-transition point $y\in\gamma$ such that $d(w, y) <
D_2$. Hence, $d(x, y) < d(x, ho) +d(ho, zo)+d(zo, w)+d(w, y) <
L+D_1+D_2+2R$. So (2) is proved.
\end{proof}

\section{Small cancellation in relatively hyperbolic groups}\label{Section8}
The notion of a rotating family was due to Gromov, and further developed by Coulon in \cite{Coulon2} and Dahmani-Guirardel-Osin in \cite{DGO}.
In this section, we make use of the theory of a very rotating family to study
small cancellation quotients 
$$G \to G/\nh$$
of a relatively hyperbolic group $G$ over a hyperbolic element $h$, and to find an embedded set of elements in the
quotient $G/\nh$. 
The key tool here is Qualitative Greendlinger Lemma in \cite{DGO}.
However, this lemma lives in a cone-off space, and in order to use
it, we need a notion of lift paths to transfer back information to the
original space.
\subsection{Rotating family and hyperbolic cone-off}

Assume that $(\dot X, \dot d)$ is a $\delta$-hyperbolic space, on which
$G$ acts (usually non-properly) on $\dot X$ by isometries.
\begin{defn}
Let $A$ be a $G$-invariant set in $\dot X$ and $\{G_a: a \in A\}$ a
collection of subgroups in $G$ such that $G_a(a)=a$ and $gG_ag^{-1}
= G_{ga}$ for any $a \in A$. We call the pair $\mathcal A=(A,
(G_a)_{a \in A})$ \textit{a rotating family}.
\end{defn}

We are particularly interested in a \textit{very rotating} family: informally speaking, this requires $A$ to be
sufficiently separated and each $G_a$ rotates at $a$ with a very
``large" angle. See  \cite{DGO} and \cite{Coulon2} for a precise definition. The
following result in \cite[Lemma 5.16]{DGO} is a qualitative version of Greendlinger Lemma. 

\begin{lem}[Qualitative Greendlinger Lemma]\label{greendlinger}
Let $\dot X$ be a $\delta$-hyperbolic geodesic metric space with a
very rotating family $\mathcal A=(A, (G_a)_{a\in A})$. Let $g \in
\langle \cup_{a \in A} G_a\rangle \setminus 1< G$. Then for any $o \in
\dot X \setminus N_{20\delta}(A)$, there exists $a \in A \cap [o,
go]$ and $h \in G_a, q_1, q_2 \in [o, go]$ satisfying $40\delta
<\dot d(q_1, q_2) <50\delta$ and $\dot d(q_1, hq_2) < 8\delta$.
\end{lem}

Let $(Y, d_Y)$ be a geodesic metric space. For a fixed $r_0 \ge 0$, define $\mathcal C(Y)$ to be the quotient of $Y \times [0, r_0]$ by collapsing $Y\times
{0}$ to a point. The \textit{apex} $a(Y)$ and \textit{base} are the
images of $Y \times 0$ and $Y\times r_0$  in $\mathcal C(Y)$
respectively. If a group $G$ acts on $Y$ by isometries, then $G$
acts naturally on $Y \times [0, r_0]$ by $g(y, r) = (gy, r)$. This
action descends to $\mathcal C(Y)$ with a fixed point $a(Y)$.


We endow $\mathcal C(Y)$ with a geodesic metric denoted by
$d_{\mathcal Y}$, whose precise definition is not relevant here but has the
following consequence.
\begin{lem}\label{conemetric}\cite[Part I, Ch. 5]{BriHae}
Let $(Y, d)$ be a metric space and $r_0 >0$. Then the following
hold.
\begin{enumerate}
\item
$d_{\mathcal Y}((y, r), (y', r')) \le 2 r_0$ for any two $(y, r)$ and
$(y', r')$ in $\mathcal C(Y)$.
\item
A geodesic between $(y, r_0)$ and $(y', r_0)$ passes the apex in
$\mathcal C(Y)$ if and only if $d(y, y') \ge \pi \sinh r_0$.  
\end{enumerate}
\end{lem}


A major source of rotating families are provided by a so-called \textit{cone-off} construction over a space relative to subspaces.  Assume that $G$ acts on a $\delta$-hyperbolic space $(X, d)$ with a $G$-invariant $10\delta$-quasi-convex subspaces $\mathbb Q$.  The \textit{cone-off} $\dot X(\mathbb Q)$ over $X$
relative to $\mathbb Q$ is the quotient of the disjoint union $$X
\sqcup (\sqcup_{Q \in \mathbb Q}\; \mathcal C(Q))$$ by
identifying the base of $\mathcal C(Q)$ with $Q$ in $X$.

By exhausting the space $X$ by $G$-orbits, we can use Zorn Lemma to construct a nerve graph quasi-isometric to $X$ on which $G$ acts. Note that the results proven in this section is stable up to such a quasi-isometry. So we assume that $X$ is a graph for simplicity. In particular, $\dot X$ with the natural length metric $\dot d$ is a geodesic space.

If $\mathbb Q$ has the bounded intersection property, then
$\dot X$ is a hyperbolic space, cf. \cite[Thm 3.5.2]{Coulon2}. Denote by $A(\mathbb Q) = \{a(Q): Q \in \mathbb Q\}$ the set of
apices. The stabilizers of each $Q \in \mathbb Q$ together
$A(\mathbb Q)$ gives rise to a rotating family. 

In what follows, we consider
a particular example arising from a hyperbolic element $h \in G$, when $G$ acts properly on $X$.

The \textit{axe} $\ax(h)$ of $h$ is the $2\delta$-neighborhood of
the union of all geodesics between two fixed points  of $h$ on
$\pX$. Then $\ax(h)$ is $2\delta$-quasi-convex in $X$ so that we can form a contracting system as follows
\begin{equation}\label{Qh}
\mathbb Q = \{g \ax(h): g \in G\}.
\end{equation} It is well-known that $\langle h \rangle$ is of finite index in the group
\begin{equation}\label{Eh}
E(h) = \{g \in G: \exists n \in \mathbb Z\setminus \{0\}, g h^n g^{-1} = h^{\pm n}
\}.
\end{equation}
which admits a proper and
co-compact action on $\ax(h)$.
 Denote 
\begin{equation}\label{EE}
\mathbb E =\{g E(h): g\in G\}
\end{equation}
As discussed in Example
\ref{contractingsystem}, $\mathbb U \cup \mathbb Q$ and $\mathbb P \cup \mathbb E$ are both  contracting systems with the bounded intersection property.


\begin{lem}\label{rotatinghypelem}\cite[Proposition 6.23]{DGO}
There exists universal real numbers $\delta_0>0, r_0 >20\delta_0$
such that for any hyperbolic element $h$ in $G$, there exist
$k=k(h), \iota=\iota(h)>0$ with the following property.

Consider the cone-off $\dot{X_\iota}(\mathbb Q)$ over the
scaled metric space $X_\iota = (X, \iota \cdot d)$. For any $n\ge 1$, let 
\begin{equation}\label{Hn}
\mathcal H_n = \{g \langle h^{nk} \rangle g^{-1}:  g \in G \}.
\end{equation}
Then $(A(\mathbb Q), \mathcal H_{n})$ is a very rotating family
on the $\delta_0$-hyperbolic space  $\dot{X_\iota}(\mathbb Q)$.
\end{lem}

If $G$ is a relatively hyperbolic group, then the quotient
$\bar G = G/\llangle h^{nk}\rrangle$ is relatively hyperbolic for any
$n \ge 1$. See \cite{DGO} for more detail.

\subsection{Lift paths in a cone-off space}
Let $(X, d)$ be a hyperbolic space and $\mathbb Q$ a locally finite collection of uniformly quasi-convex subspaces with the bounded intersection. We consider another version of a lift path in the cone-off space $\dot X(\mathbb Q)$,
analogous to that of a lift path in Section \ref{Section6}. The idea is, similarly, to replace each part of $\gamma$ in the cone $\mathcal C(Q)$ of $Q\in \mathbb Q$ by a geodesic in $X$.

\begin{defnconst}[Lift path in $\dot X$]
Let $\gamma$ be a geodesic in $\dot X(Q)$ with both endpoints in $X$. The
\textit{lift path} of $\gamma$ is defined as follows.

Assume that $\gamma\cap \mathcal C(Q)  \neq \emptyset$ for some 
$Q \in \mathbb Q$, and let $a_-, a_+ \in Q$ be the corresponding entry and
exit points of $\gamma$ in $\mathcal C(Q)$.  Note that $\mathcal C(Q)\cap \mathcal C(Q')=Q\cap Q'$ for distinct $Q, Q' \in \mathbb Q$. Thus, $[a_-, a_+]_\gamma$ and $[a'_-,
a'_+]_\gamma$ intersect, if at all, only in their endpoint.  We replace
$[a_-, a_+]_\gamma$ by some geodesic in $X$ with same endpoints.

Repeating the above replacement for all such $Q\in \mathbb Q$ gives the lift path
$\hat \gamma$ in $X$.
\end{defnconst}

We have an analogue to Lemma \ref{lift} by the same proof.
\begin{lem}\label{liftconeoff}
There exist $\lambda=\lambda(r_0), c=c(r_0)\ge 0$ such that  for any geodesic $\gamma$  in $\dot X(\mathbb Q)$ with both endpoints in $X$,  the lift $\hat \gamma$  is a $(\lambda,
c)$-quasi-geodesic in $X$.
\end{lem}

The following technical lemma will be used in Lemma \ref{kerelem} .
\begin{lem}\label{liftconeoff2}
There exist $\epsilon, R, L, D \ge 0$ depending on $r_0$ with the
following property.

Let $\gamma$ be a geodesic in $\dot X(\mathbb Q)$ with
$\gamma_-, \gamma_+\in X$. Consider $Q \in \mathbb Q$ such that $Q \cap
\gamma \neq \emptyset$ with the entry and exit
points $a_-, a_+$   in $\mathcal C(Q)$ respectively.  Assume
that $d(a_-, a_+) > L$. Then there exist $(\epsilon, R)$-transition
points $x, y$ in $\beta=[\gamma_-, \gamma_+]$ in $X$ such that $d(x,
a_-) < D$ and $d(y, a_+) < D$.
\end{lem}
\begin{proof}
Let $\lambda, c$ be given by Lemma \ref{liftconeoff} so that  $\hat
\gamma$ is a $(\lambda, c)$-quasi-geodesic.   By the stability of quasi-geodesics, there exists $\epsilon>0$ such that $\hat \gamma \subset N_\epsilon(\beta)$.

Let $Q \in \mathbb Q$ be given in the lemma.  Assume that $L
>\pi\sinh r_0$ and then $a(Q) \in \gamma$. Let $z, w$ be the corresponding  entry and exit
points of $\hat \gamma$ in the $\epsilon$-neighborhood $N_\epsilon(Q)$ of $Q$ in $X$.  Project $z$ to $z'\in Q$ so that  $d(z, z') \le \epsilon$. By Lemma \ref{conemetric}, $\dot d(z', a_-) \le  2r_0$ and then $\dot d(z, a_-) \le \epsilon+2 r_0$. 

Consider  the geodesic $p:=[z, a_-]_\gamma$ between $z, a_-$ in $\dot X$.  Note that  if a geodesic in $\dot X$ with endpoints in $X$ passes an apex, then its $\dot d$-length is of at least $2r_0$. Moreover, since $\mathbb Q$ is locally finite in $(X, d)$, it follows that there are at most finitely many, say $N\ge 1$,   $\{Q_1, \cdots, Q_N\}\subset \mathbb Q$ intersecting $p$. 

Let $D_0$ be given by Lemma \ref{contracting}.(3), for which we also assume that $\mathbb Q$ has $D_0$-bounded projection. The parts of $p$ inside $X$ has $d$-length at most $\epsilon+2 r_0$. Consider the entry and exit points $(Q_i)_-, (Q_i)_+$ of $p$ in $Q_i$ ($1\le i\le N$). We  project to $Q_i$ the other $Q_j$, $Q$, the geodesic $[z, z']$  and parts of $p$ inside $X$: 
$$
\begin{array}{lll}
d((Q_i)_-, (Q_i)_+) &\le \sum_{j\ne i} \proj_{Q_i}(Q_j) +\proj_{Q_i}(Q) + \proj_{Q_i}([z, z'])+\proj_{Q_i}(p\cap X)\\
&\le (N+2)\cdot D_0 + 2\epsilon+2 r_0 
\end{array}
$$
which follows by the $D_0$-bounded projection and Lemma \ref{contracting}.(3). In total, 
\begin{equation}\label{EQza}
d(z, a_-) \le \sum_{1\le i\le N} d((Q_i)_-, (Q_i)_+) + \epsilon+2 r_0 \le (N+1)((N+2)\cdot D_0 + 2\epsilon+2 r_0).
\end{equation} Similarly, the same upper bound holds for $d(w, a_+)$.

\begin{figure}[htb]
\centering \scalebox{0.5}{
\includegraphics{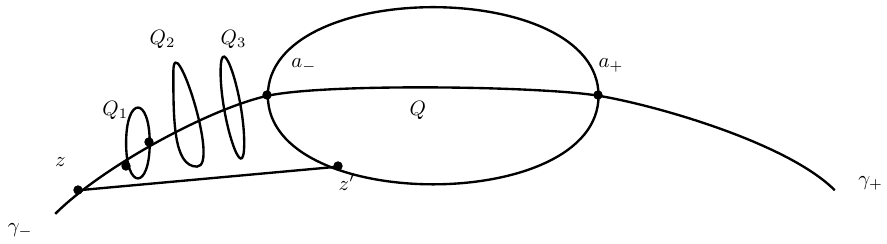} 
} \caption{Lemma \ref{reduelem}} \label{Figure4}
\end{figure}

Let $x, y$ be the entry and exit points of $\beta$ respectively in
$N_\epsilon(Q)$. Since $Q$ is quasi-convex, we assume that $[x,y]_\beta \subset N_\epsilon(Q)$. Assume, in addition, that $\mathbb Q$ is $(\epsilon, D_0)$-contracting for $(\lambda, c)$-quasi-geodesics.   We then project the parts of $\gamma$ and $\beta$ outside $N_\epsilon(Q)$ to see that $d(z, x), d(w, y) <
2D_0+2\epsilon$. By (\ref{EQza}), we have  
\begin{equation}\label{EQaxy}
d(a_-, x), d(a_+, y) <  D,
\end{equation}
where
$D:=(N+2)((N+2)\cdot D_0 + 2\epsilon+2 r_0)$.

Let $R=\sup_{Q\in \mathbb Q, U\in \mathbb U} \{\diam{N_\epsilon(Q)
\cap N_\epsilon(U)} \} < \infty$. It suffices to show that $x$ is an
$(\epsilon, R)$-transition point in $\gamma$.

Suppose, to the contrary, that $x$ is $(\epsilon, R+1)$-deep in a horoball
$U \in \mathbb U$. We can assume that $L > 2D + R+1$ so that $d(x,
y) >R$ by (\ref{EQaxy}). Since $[x, y]_\beta \subset N_\epsilon(Q)$,  we have $N_\epsilon(Q) \cap
N_\epsilon(U)$ has diameter at least $R$. This contradicts to the choice of $R$. Therefore,
$x$ is an $(\epsilon, R)$-transition point in $\beta$. The same holds for $y$.
\end{proof}

\subsection{Finding embedded subsets in quotients}

We shall construct a subset of $G$ which injects into $G/\nkh$ for
$n\gg 0$. This will be explained by a series of lemmas.

Let $\mathbb Q,\; \mathbb E,\;\mathcal H_n$ be in (\ref{Qh}), (\ref{EE}), (\ref{Hn}) respectively. By Lemma \ref{rotatinghypelem}, for any $n >1$, $(A(\mathbb Q),
\mathcal H_n)$ is a very rotating family on $\dot{X}_\iota(\mathbb
Q)$. In what follows, we omit the scaling factor $\iota$ for notational simplicity and denote $\dot{X}=\dot{X}_\iota(\mathbb
Q)$, as the
reasoning for $(X, d)$ is valid for $(X, \iota \cdot d)$ upon a scaling
of the metric $d$.

Fix $o \in X \setminus (\mathcal U \cup \mathcal Q)$, where
$\mathcal Q=\cup_{Q\in \mathbb Q} Q$. We first prove that kernel
elements shall quasi-contain some high power of the
relator $h^{k}$, where $k=k(h)$ is given by Lemma \ref{rotatinghypelem}.

\begin{lem}\label{kerelem}
There exist constants $\epsilon =\epsilon(h), k =k(h)>0$ and a function $L=L(h, n) >0$ for $n >1$ with $L(h, n)\to \infty$ as $n\to \infty$ such that the
following hold for any
$g\in \llangle h^{nk}\rrangle\setminus 1$:  
\begin{enumerate}
\item
any geodesic $\gamma=[o, go]$ in $X$ contains an $(\epsilon, L)$-deep point in some $Q\in \mathbb Q$.

\item
any geodesic $\gamma=[1, g]$ in $\Gx$ contains an $(\epsilon, L)$-deep point in some $E\in \mathbb E$.
\end{enumerate}

\end{lem}
\begin{proof}
For a subgroup $H$ in $G$, consider two functions $$\inj_G(H) := \inf\{d(h, 1):  
h \in H \}, \; \inj_X(H) := \inf\{d(ho, o):  
h \in H \}.$$


\textbf{(1)}. Let $\beta=[o, go]$ be a geodesic in $(\dot X, \dot
d)$. By Lemma \ref{greendlinger}, there exists $a=a(Q) \in A \cap
\beta$ such that there exists $f' \in G_a\setminus \{1\}, q_1, q_2 \in \beta$ with
$40\delta_0 <\dot d(q_1, q_2) <50\delta_0$ and $\dot d(q_1, f'q_2) <
8\delta_0$. For definiteness, assume that $G_a=g_a \langle
h^{n\cdot k}\rangle g_a^{-1}$  for some $g_a\in G$ and write
\begin{equation}\label{fexpression}
f' =g_a  h^{n\cdot k\cdot i} g_a^{-1}
\end{equation}
for some $i \in \mathbb Z$.

Let $q'_1, q'_2$ be the radical projections of $q_1, q_2$ to the
base of $\mathcal C(Q)$ respectively. Thus, $a(Q) \notin [q_1,
f'q_2]_\beta$. Indeed, if $a(Q) \in [q_1, f'q_2]_\beta$, then $40\delta_0 < \dot d(q_1, q_2)=\dot d(q_1, f'q_2) < 8\delta_0$, where the equality follows since $f'$ rotates around $a$. This is a contradiction.
Hence, $a(Q) \notin [q_1, fq_2]_\beta$ and then $a(Q) \notin [q'_1,
fq'_2]_\beta$.  By Lemma \ref{conemetric}, we have 
\begin{equation}\label{shortdist}
d(q'_1, f'q'_2)
< 2 r_0.\end{equation}


Note that $q'_1, q'_2$ are entry and exit points of
$\beta$ in $\mathcal C(Q)$ respectively, which lie in  the lift of $\beta$. By Lemma \ref{liftconeoff} the lift of $\beta$ is a quasi-geodesic so that by the stability of quasi-geodesics there exist $D_1>0$ and $x, y$ in $\gamma=[o, go]$ such that $d(q_1', x), d(q'_2, y) < D_1$. By  (\ref{shortdist})
we have
\begin{equation}\label{fexpression2}
d(f'q'_2, x),\;  d(q'_2, y) < 2 r_0 + D_1.
\end{equation}

Since $E(h)$ acts co-compactly on $Q=\ax(h)$,  there exist $M >0$ and
$j>0$ such that $d(g_ah^{j} o, q'_2) < M$.  Hence by
(\ref{fexpression}) and (\ref{fexpression2}),
\begin{equation}\label{fexpression3}
d(g_a h^j\cdot  h^{n\cdot k\cdot i} o, x), \;d(g_ah^{j} o, y) < \epsilon_1 :=2 r_0 + D_1 + M.
\end{equation}
Since $Q$ is quasi-convex, we can assume that $[x, y]_\gamma\subset N_{\epsilon_1}(g_aQ)$ by increasing $\epsilon_1$.  This implies that $$\diam{\gamma\cap N_{\epsilon_1}(g_aQ)} \ge d(h^{n\cdot k\cdot i}o, o)-2\epsilon_1\ge L_X(h,n):=\inj_X(\langle h^{n k}\rangle)-2\epsilon_1.$$  Then $\gamma$  contains an $(\epsilon_1, L)$-deep point in $g_a Q\in \mathbb Q$ for $L=L_X(h,n)$.

\textbf{(2)}. Since $d(q_1', x), d(q'_2, y) < D_1$, by   (\ref{fexpression3}), we have $$d(q'_1, q'_2)\ge d(h^{n \cdot k\cdot i} o,   o)- 2\epsilon_1-2D_1.$$ Let $\epsilon, R, L, D_2>0$ be given by Lemma \ref{liftconeoff2}. By taking a constant multiple, we choose $k=k(h)$ large enough such that $$d(h^{n \cdot k\cdot i} o,   o)> \epsilon_1+ D_1+L.$$ Then $d(q'_1, q'_2)\ge L$ so that we can apply Lemma \ref{liftconeoff2} to $q_1', q_2'$. There exist $(\epsilon, R)$-transition points $z, w\in \gamma$ such that $d(q_1', z), \;d(q_2', w)\le D_2$.

Let $\alpha=[1, g]$ be a word geodesic in $\Gx$. For a constant
$D_3=D(\epsilon, R)>0$ given by Lemma \ref{PairTransX},  there
exist $g_1, g_2\in \alpha$ such that $d(g_1 o, z), d(g_2 o, w) <
D_3$. It follows by (\ref{fexpression3}) that
\begin{equation}\label{fexpression4}
d(g_a h^j\cdot  h^{k\cdot i} o, g_1o), d(g_ah^{j} o, g_2o) < \epsilon_1 + D_1+D_2+D_3.
\end{equation}
Since $G$ acts properly on $X$, set $\epsilon_2=\max\{d_S(1, g): d(go, o)
< \epsilon_1+D_1+D_2+D_3\} < \infty$. 
This implies that $$\diam{[1, g]\cap N_{\epsilon_1}(g_aE(h))} \ge d(h^{n\cdot k\cdot i}, 1)-2\epsilon_2\ge L_G(h, n):= \inj_G(\langle h^{n\cdot k}\rangle)-2\epsilon_2.$$ Then $[1, g]$ contains an $(\epsilon_2, L)$-deep point in $g_aE(h)$ for $L=L_G(h,n)$. This concludes the proof.
\end{proof}

We are able to prove Proposition \ref{ThmBSharp} via the following useful corollary:
\begin{lem}\label{QuotientP}
For any $P\in \mathcal P$, we have that $P\cap \nkh$ is trivial for large enough $n\gg  1$. In particular, $\pi: G\to G/\nkh$ is injective  on $P$. 
\end{lem}
\begin{proof}
Let $\epsilon=\epsilon(h), k=k(h), L=L(h, n)$ be given by Lemma \ref{kerelem}. Let $1\ne g\in P\cap \nkh$. Up to a conjugation we assume that $P\in \mathcal {\bar P}$, where $\mathcal {\bar P}$ is a complete set of conjugacy representatives of $\mathcal P$. Recall that $\mathbb P=\{gP: P\in \mathcal {\bar P}\}$ and $\mathbb P\cup \mathbb E$ has bounded intersection. So there exists a constant $L_0$ depending on $\epsilon$ such that if a geodesic $\gamma=[1,g]$ contains an $(\epsilon, R)$-deep point in some $E\in \mathbb E$, then $R\le L_0$. By taking $n$ large enough such that $L(h,n)>L_0$ we get a contradiction with Lemma \ref{kerelem} (2).
\end{proof}

\begin{proof}[Proof of Proposition \ref{ThmBSharp}]
Since $G$ has no parabolic gap property, there exists $P\in \mathcal P$ such that $\g G=\g P$. By Lemma \ref{QuotientP},  the quotient map $P\to \bar P:=P/(P\cap \nkh)$ induced by $\pi: G\to G/\nkh$ is injective. Thus, $\g P=\g{\bar P}$ and so $\g{\bar P} \le \g {\bar G} \le \g G=\g P$: $\bar G$ has no parabolic gap property.
\end{proof}

Let $\Delta(v_0, v_1, v_2)$ denote (a choice of) a geodesic triangle
in $\Gx$ with vertices $v_i$. The following relative version of the thin-triangle
property in \cite{DruSapir} is useful in Lemma \ref{reduelem}.

\begin{lem}\label{relthintri}
There exists $\sigma>0$ such that for any geodesic triangle
$\Delta(v_0, v_1, v_2)$,  one of the following two cases holds
\begin{enumerate}
\item
there exists $v\in G$ such that $d(v, [v_i, v_{i+1}]) < \sigma$ for
$i\mod3$,
\item
there exists a unique $gP \in \mathbb P$ such that $d(gP, [v_i,
v_{i+1}]) < \sigma$ for $i \mod 3$.
\end{enumerate}
In the second case denote by $(v_i)_-, (v_i)_+$ the entry points of
$[v_i, v_{i-1}]$ and $[v_i, v_{i+1}]$ respectively in $N_\sigma(gP)$
for $i \mod 3$. Then $d((v_i)_-, (v_i)_+) < \sigma$.
\end{lem}

\begin{lem}\label{reduelem}
There exist $\epsilon =\epsilon(h), k =k(h)>0$ and a function $L=L(h, n) >0$ for $n >1$ with $L(h, n)\to \infty$ as $n\to \infty$ such that the
following holds. Assume that $g_2 = g_1 g$ for $g_1, g_2\in G, g \in \llangle h^{nk}
\rrangle$ for $n\gg 1$.  
\begin{enumerate}
\item
Either $[o, g_1o]$ or $[o, g_2o]$ contains an $(\epsilon, L)$-deep point in some $Q\in \mathbb Q$.
\item
Either $[1, g_1]$ or $[1, g_2]$ contains an $(\epsilon, L)$-deep point in some  $E\in \mathbb E$.
\end{enumerate}
\end{lem}
\begin{proof}
We only show the statement (2). The argument for (1) is simpler,
with Lemma \ref{relthintri} replaced by the thin-triangle property in hyperbolic spaces. 

Consider the geodesic triangle $\Delta(1, g_1, g_2)$ in $\Gx$ with
sides $\gamma_1=[1, g_1], \gamma=[g_1, g_2],
\gamma_2=[1, g_2]$, where $g_2=g_1g$.

Let $\epsilon, L_G(h, n)$ given by 
Lemma \ref{kerelem} so that $\gamma$ contains an $(\epsilon, L_G(h, n))$-deep point in some $E(h)$-coset $E\in \mathbb E$.  So, $\diam{\gamma \cap N_{\epsilon} (E)} >  L_G(h, n)$. Let $x, y$ be the entry and exit points of $\gamma$ in $
N_{\epsilon}(E)$ respectively. Then
\begin{equation}\label{dxy}
d(x, y) > L_G(h,n)-2\epsilon.
\end{equation}
Let $\sigma>\epsilon$ be the constant given by Lemma \ref{relthintri} so that we have the
following two cases.

\textbf{Case 1}. There exists $z \in \gamma$ such that $d(z,
\gamma_i) < 2\sigma$ for $i=1, 2$. Let $z' \in \gamma_1$ so that $d(z, z') < 2\sigma$. If   $z
\in [x, y]_\gamma$, we assume by (\ref{dxy}) that
\begin{equation}\label{dzx}
d(z, x)>L_G(h,n)/2-\epsilon,
\end{equation} for concreteness. If $z \notin [x, y]_\gamma$,  assume for definiteness that $z\in [y, \gamma_+]_\gamma$. In both cases we claim that $[g_1, z']_{\gamma_1}\cap N_\epsilon(E)\ne \emptyset$. We verify  the claim only for the case $z\in [y, \gamma_+]_\gamma$; the other case is similar.   Without loss of generality we assume that $\mathbb P \cup \mathbb E$ is an
$(\epsilon, D)$-contracting system for the same $\epsilon$ and a constant $D$ also satisfying Lemma \ref{contracting}.(3). If $[g_1, z']_{\gamma_1}\cap N_\epsilon(E)= \emptyset$,  we get by projection that 
$$
\begin{array}{lll}
d(x, y) &\le \diam{\proj_{E}([g_1, z']_{\gamma_1})}+\diam{\proj_{E}([z',z])} + \diam{\proj_{E}([y, z]_\gamma)}+\diam{\proj_{E}([x, g_1]_\gamma)}\\
&\le 4D+2\sigma,
\end{array}
$$ where $[z',z]$ is of length at most $2\sigma$.  Choose $n>1$ such that $L_G(h,n) \ge 4D+2\sigma+2\epsilon$. We then get a contradiction with (\ref{dxy}). See Figure \ref{Figure2}.

\begin{figure}[htb]
\centering \scalebox{0.5}{
\includegraphics{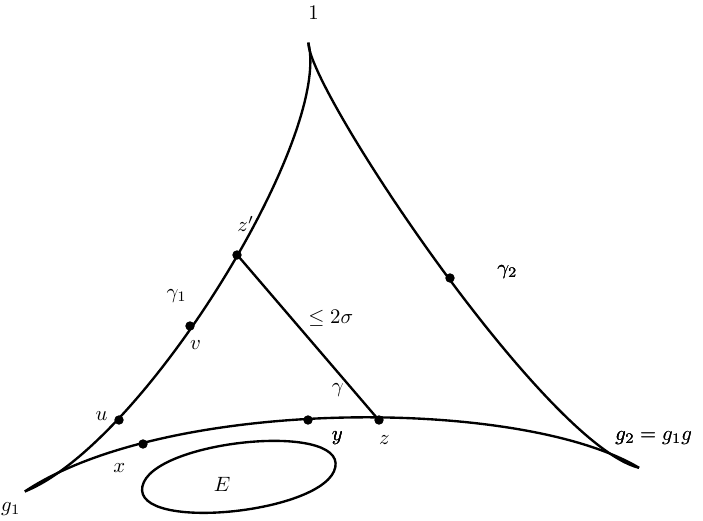} 
} \caption{Lemma \ref{reduelem}} \label{Figure2}
\end{figure}

Choose $u, v\in \gamma_1$ the respective entry and exit points of $[g_1, z']_{\gamma_1}$ in $N_{\epsilon}(E)$. We project the parts of $[g_1, z']_{\gamma_1}$ and $[g_1, z]_\gamma$ outside $N_{\epsilon}(E)$ to $E$. By the contracting property of $E$ we have $d(x, u), d(v, y)\le 3D+2\sigma$. We get by (\ref{dxy}) and (\ref{dzx}): 
$$\begin{array}{lll}
\diam{\gamma_1
\cap N_{\epsilon}(E)} &\ge d(u, v)>d(x, y) -2(3D+2\sigma) \\
&> L(h,n):=L_G(h, n)/2-2(3D+2\sigma).
\end{array}$$ Since $E$ is quasi-convex, we can assume that $[u, v]_{\gamma_1}\subset N_{\epsilon_1}(E)$ for the same $\epsilon$. Then $\gamma_1$ contains an $(\epsilon, L)$-deep point in $E$ for $L=L(h,n)$.

\textbf{Case 2}. Let $gP \in \mathbb P$ be given by Lemma
\ref{relthintri}, and $u, v$ the entry and exit points of $\gamma$
in $N_\sigma(gP)$ respectively so that $d(u,
\gamma_1), d(v, \gamma_2) < \sigma$.

As $\mathbb P \cup \mathbb E$ has the bounded intersection property, there exists
a constant $C=C(\epsilon, \sigma)>0$ such that $\diam{[x,
y]_\gamma \cap [u, v]_\gamma} < C$. Without loss of generality,
assume that $d(y, u) < C$, with the case that $d(x, v) < C$ being symmetric.

By a similar argument as above in Case 1, we can find a function $L(h, n)$ such that $\gamma_1$ contains an $(\epsilon, L)$-deep point in some left $E(h)$-coset in $\mathbb E$.
\end{proof}

For $\epsilon, L \ge 0$, denote by
$\E$ the set of elements $g$ in $G$ such that $g$ contains an $(\epsilon, L)$-deep point in some $E\in \mathbb E$. Similarly,  denote by $\Q$ the set of $go$ in $X$ such that some $[o, go]$ contains an
$(\epsilon, L)$-deep point in some $Q\in \mathbb Q$.

We are now in a position to state the main result of this subsection.

\begin{prop}\label{injective}
Let $\epsilon =\epsilon(h),k =k(h)$ and $L=L(h, n)$ for $n\ge 1$ be as in Lemma
\ref{reduelem}. Then
\begin{enumerate}
\item
The map $\pi: G \to G/\nkh$ is injective on $\GE$.
\item
The map $\pi: X \to X/\nkh$ is injective on $\GQ$.
\end{enumerate}
\end{prop}
\begin{proof}
We prove the statement \textbf{(2)}.  Let $g_1o, g_2o \in \GQ$
such that $\pi(g_1o)=\pi(g_2o)$. Then $\nkh g_1o = \nkh g_2 o$ and so there exists $g \in \llangle h^{kn} \rrangle$ such that $g_1 o =
g_2 go$. By Lemma \ref{reduelem}, 
at least one of $[o, g_1o], [o, g_2o]$ contains an $(\epsilon, L)$-deep point in some $Q\in \mathbb Q$. This contradicts to the choice of $g_1o, g_2o \in
\GQ$.
\end{proof}

\section{Proofs of Theorems \ref{ThmA} \& \ref{ThmB}}\label{Section9}
We are ready to give proofs of main theorems. Let $\GP,\; X, \;\mathbb U, \;\mathcal U$ be as in Definition
\ref{RHdefn}. Fix a basepoint $o \in X \setminus \mathcal U$. Let
$h$ be a hyperbolic element with the corresponding $\mathbb Q$ and $\mathbb E$  defined in (\ref{Qh}) (\ref{EE}). We have that $\mathbb P \cup  \mathbb
E$ is a contracting system with the bounded intersection property in $\Gx$, and $\mathbb U \cup \mathbb Q$ is so in $X$.  Denote by $\epsilon_0>0$ the constant given by Lemma \ref{uniformtrans} for both  
contracting systems $\mathbb P \cup \mathbb E$ and $\mathbb U \cup \mathbb Q$.

\subsection{Proof of Theorem \ref{ThmA}}
Fix $0<\sigma<\g G$.  We consider transition points relative to $\mathbb P \cup
\mathbb E$. Let $\epsilon>\epsilon_0, r, R$ be given by Theorem \ref{LargeTreeG}. There exists a geodesic
tree $\mathcal T$ with $\g{\mathcal T}
> \sigma$ and for any geodesic $\gamma$ in $\mathcal T$ and $x \in \gamma$, there exists an $(\epsilon,
R)$-transition point $y \in \gamma$ such that $d(x, y) < r$.

Let $\epsilon'=\epsilon'(h), k=k(h), L(h, n)$ be given by Lemma \ref{injective}.  Choose $n$ large enough such that $$L:=L(h, n) >2(L(r+R+1)+\epsilon'),$$ where $L(r+R+1)$ is given by Lemma
\ref{uniformtrans}.  Then $\pi:
G\to G/\nkh$ is injective in $\GEE$.  

We claim
that $\mathcal T^0 \subset \GEE$. Suppose, to the contrary, that there exists $g \in \mathcal T^0 \cap \EE$.  Since $[1,g]$ contains an $(\epsilon', L)$-deep point in $E\in \mathbb E$, by Lemma \ref{uniformtrans}, $[1,g]$ contains an
 $(\epsilon_0, r+R+1)$-deep point $v$ in $E$. Thus, any point in $B(v, r)\cap [1,g]$ is $(\epsilon_0, R)$-deep in $E$. This contradicts to the  property of $\mathcal T$ stated above. Hence, $\mathcal T^0 \subset \GEE$.

As $\pi$ is length-decreasing, it follows that $\g {\mathcal T} \le
\g{\GEE}$. The proof
is complete.

\subsection{Proof of Theorem \ref{ThmB}}
The proof is analogous to the previous one.

Recall that
$\bar G = G /\nkh$ acts naturally on $\bar X= X/\nkh$ by $\bar g \cdot
\bar x= \bar{gx}$. Equip $\bar X$ with the metric $\bar d(\bar x, \bar
y) = d(\nkh x, \nkh y)$ so that $\bar G$ acts isometrically and properly
on $\bar X$.

Fix $0<\sigma<\g G$. We consider the notion of transition points relative to $\mathbb U
\cup \mathbb Q$. Let $\epsilon>\epsilon_0, r, R>0$ be given by Theorem \ref{LargeTreeX}.   There exists a geodesic
tree $\mathcal T$ in $\Gx$ such that $\g{\mathcal To}
> \sigma$ and the following property holds. Consider a geodesic $\gamma=[o,go]$ in $X$ for $g \in \mathcal T^0$. Then for any $x \in
\gamma$, there exists an $(\epsilon, R)$-transition point $y \in
\gamma$ such that $d(x, y) < r$.

Let $\epsilon'=\epsilon'(h), k=k(h), L(h,n)$ be given by Lemma \ref{injective}.
Choose $n$ large enough such that $L:=L(h, n) >2(L(r+R+1)+\epsilon')$. Then $\pi:
X\to \bar X$ is injective in $\GQQ$. As in the proof
of Theorem \ref{ThmA}, we obtain that $\mathcal T^0 o \subset \GQQ$.

As $\pi$ is length-decreasing, it follows that $\g {\mathcal To} \le
\g{\GQQ}$. Hence, $\g{\bar G} \ge \g {\mathcal To} \ge \sigma$. 

We now prove that $\bar G$ has parabolic gap property for $n\gg 0$. Indeed,  $P$ is sent by the map $\pi$ to $\bar P:=P/(\nkh\cap P)$. By Lemma \ref{QuotientP} we have $\bar P=P$ and $\g P=\g {\bar P}$. So for $n\gg 1$ we have $\g {\bar G} > \g P=\g{\bar P}$ and thus $\bar G$ has PGP property. Theorem
\ref{ThmB} is proved.

\appendix
\section{Uniqueness and ergodicity of PS-measures}\label{appendixA}
This appendix is devoted to the proof for the uniqueness and
ergodicity of PS-measures (in fact, for any quasi-conformal density
without atoms). The results proven here will not be used in
this paper.

Let $B$ be a metric ball of radius $rad(B)$ in a metric space $X$.
For $t >0$ define $tB$ to be the union of all balls of radius
$t\cdot rad(B)$ intersecting $B$. Let's recall the
following covering result.
\begin{lem}\label{covering}\cite[Theorem 2.1]{Matt}
Let $X$ be a proper metric space and $\mathcal B$ a family of balls
in $X$ with uniformly bounded radii. Then there exists a subfamily
$\mathcal B' \subset \mathcal B$ of pairwise disjoint balls such
that the following holds
$$
\bigcup\limits_{B \in \mathcal B} B \subset \bigcup\limits_{B \in
\mathcal B'} 5B.
$$
\end{lem}

Let $\mathcal B=\{B_{\rho_1}(\xi, r_i): \xi \in \cG\}$ be a
collection of balls centered at conical points $\xi$, with radii
$r_i$ provided by Lemma \ref{ombredisks}.  We consider a Hausdorff-Caratheodory measure $H_\sigma(, \mathcal
B)$ relative to $\mathcal B$ and the gauge function $r^\delta$. The
Hausdorff measure in the usual sense is to take $\mathcal B$ as the
set of all metric balls in $\pG$.

\begin{lem}\label{Unique}
Let $\{\mu_v\}_{v\in G}$ be a $\sigma$-dimensional quasi-conformal density without
atoms on $\pG$. Then we have
$$
H_{\sigma}(A, \mathcal B) \asymp \mu_1(A).
$$
for any subset $A \subset \pG$. In particular, $\mu$ is unique in
the following sense: if $\mu, \mu'$ are two such quasi-conformal density, then the
Radon-Nikodym derivative $d\mu/d\mu'$ is bounded from up and below.
\end{lem}
\begin{proof}
Let $A$ be any subset in $\pG$. Let $\mathbb B' \subset \mathbb B$
be an $\epsilon$-covering of $A$. Then $\mu_1(A) < \sum_{B \in
\mathcal B'} \mu_1(B)$. Let $\epsilon \to 0$. By Lemma
\ref{ombredisks}, we obtain that $\mu_1(A) \le C \cdot H_{\sigma}(A,
\mathcal B)$.

To establish the other inequality, we adapt an argument
of McMullen, cf. \cite[Remark 2.5.11]{Calegari}. Note that $\mu_1,
H_\sigma$ are Radon measures. Then for any $\tau
>0$ there exists a compact set $K$ and an open set $U$ such that $K
\subset A \subset U$ and $H_\sigma(U\setminus K) < \tau, \mu_1(U
\setminus K) < \tau$.

Set $\epsilon_0:=\rho_1(K, \pG \setminus U) >0$. For any $0 <
\epsilon < \epsilon_0$, let $\mathcal B_1 \subset \mathcal B$ be an
$\epsilon$-covering of $K$. Then there exists a sub-collection
$\mathcal B_2$ of $\mathcal B_1$ given by Lemma \ref{covering} such
that
$$
H_\sigma(K, \mathcal B)  <
\sum_{B \in \mathcal B_2} (5\cdot rad(B))^\sigma< C \cdot\mu_1(U),
$$
for some constant $C>0$ provided by Lemma \ref{ombredisks}. Letting
$\tau\to 0$ yields $ H_\sigma(A, \mathcal B) < C \mu_1(A)$.
\end{proof}

\begin{lem}\label{Ergodic}
Let $\{\mu_v\}_{v\in G}$ be a $\sigma$-dimensional quasi-conformal density without
atoms on $\pG$. Then $\{\mu_v\}_{v\in G}$ is ergodic with respect to the action of $G$ on $\pG$.
\end{lem}
\begin{proof}
Let $A$ be a $G$-invariant Borel subset in $\pG$ such that $\mu_1(A)
>0$. Restricting $\mu$ on $A$ gives rise to a $\g
G$-dimensional conformal density without atoms on $A$.

By Lemma \ref{Unique}, there exists $C>0$ such that for any subset
$X \subset \pG$, we have $\mu_1(X) < C \cdot H_\sigma(A \cap X, \mathcal
B)$.  Thus, $\mu_1(\pG \setminus A)=0$.
\end{proof}


By Proposition \ref{PSatom}, we have the following proposition.
\begin{prop}\label{PSMeasureG}
The PS-measures $\{\mu_v\}_{v \in G}$ is a $\g G$-dimensional quasi-conformal density
without atoms. Moreover, they are unique and ergodic.
\end{prop}

\section{Finiteness of partial cone types}
Recall that the
finiteness of cone types is established by Cannon in \cite{Cannon} for a hyperbolic
group. We shall state an analogous ``partial" result
in the relative setting.

For fixed $\epsilon, R>0$, we define two partial cones
$\Omega_{\epsilon, R}(g), \Omega_{\epsilon, R}(g')$ to be of
\textit{same type} if there exists $t \in G$ such that $t\cdot
\Omega_{\epsilon, R}(g)= \Omega_{\epsilon, R}(g')$.

\begin{lem}[Finiteness of Partial Cone Types]\label{FinPCones}
There exist $\epsilon, R_0>0$ such that for any $R > R_0$, there are
at most $N=N(\epsilon, R)$ types of all $(\epsilon, R)$-partial
cones $\{\Omega_{\epsilon, R}(g): \forall g \in G\}$.
\end{lem}
\begin{proof}
Let $\epsilon, R_0>0$ be given by Lemma \ref{floydtrans}. For any $R
> R_0$, let $\Omega_{\epsilon, R}(g)$ be a partial cone and $h$ an
element in $\Omega_{\epsilon, R}(g)$. We first make the following
claim.
\begin{claim}
Assume that $d(h, g) \ge 2R+1$ for $h\in \Omega_{\epsilon, R}(g)$. Then there exists a constant $\kappa
>0$ such that the following holds
$$
\rho_{g}(1, h) > \kappa.
$$
\end{claim}
\begin{proof}
By definition of the partial cone, there exists some geodesic $\gamma=[1, h]$ with  $g \in \gamma$ such that
$\gamma$ contains an $(\epsilon, R)$-transition point $v$ in $B(g,
2R)$.


Since $v$ is an $(\epsilon, R)$-transition point in $\gamma$, there
exists $\kappa=\kappa(\epsilon, R)$ given by Lemma \ref{floydtrans}
such that $\rho_v(1, h) > \kappa$. Note that $d(v, g) \le 2R$. By
the property (\ref{equivmetric}), we can assume that $\rho_{g}(1,
h) > \kappa$, up to divide $\kappa$ by a constant depending
on $2R$.
\end{proof}

Define $C =\varphi(\kappa)$, where $\varphi$ is given by Lemma
\ref{karlssonlem}. Let $F_g$ be the subset of elements $f\in B(g, 2C+1)$ such that $d(1, gf) \le d(1, g)$. Consider 
\begin{equation}\label{ggequal}
B_g:=B(g, 3R+1) \cap (\Omega_{\epsilon, R}(g)\cup \mathcal C(g)),
\end{equation} where $\mathcal C(g)$ is the union of all possible geodesics $[1, g]$ between $1$ and $g$.  

\begin{claim}
The sets $F_g$ and $B_g$ together determine the
type of the partial cone $\Omega_{\epsilon, R}(g)$.
\end{claim}

\begin{proof}[Proof of Claim]
Let $g, g'\in G$ such that  $F_g=F_{g'}$ and $B_g=B_{g'}$. For any $h \in
\Omega_{\epsilon, R}(g)$ and  $w := g^{-1}h$, we shall prove that $g'w \in
\Omega_{\epsilon, R}(g')$ by induction on $n=d(1, w)\ge 0$. The base case ``$n\le 3R+1$'' is true by $B_g=B_{g'}$. 

Assume $gw \in \Omega_{\epsilon, R}(g)$ and $g'w \in \Omega_{\epsilon, R}(g')$. Let $s \in S$ be
a generator such that $gws \in \Omega_{\epsilon, R}(g)$. So the following holds \begin{equation}\label{straight} d(1, gws)
=d(1, g) +d(1, w) +1.
\end{equation} We shall show that $g'ws
\in \Omega_{\epsilon, R}(g')$ by verifying the following two properties: 

\begin{itemize}
\item
the concatenation of some geodesics $[1, g'], [g', g'w]$ and $[g'w, g'ws]$ is a geodesic, and
\item
it contains an $(\epsilon, R)$-transition point in $B(g', 2R)$. 
\end{itemize} 
By assumption, the subpath $[1, g][g, gw]\cap B(g, 3R+1)$ contains an $(\epsilon, R)$-transition point. Note that the transition point is invariant by translation.  Since $B_g=B_{g'}$, the two segments $[1, g][g, gw]\cap B(g, 3R+1)$ and $[1, g'][g, g'w]\cap B(g, 3R+1)$ have the same label, and so $[1, g'][g', g'w]$ does contain an $(\epsilon, R)$-transition point in $B(g', 3R+1)$. The second property follows. So, it suffices to verify that the concatenation path is a geodesic. Suppose, by way of contradiction, that
\begin{equation}\label{notstraight}
d(1, g'ws) \le d(1, g') + d(1, w).
\end{equation} Consider a geodesic $\gamma$ between $1$ and $g'ws$. By  Induction Assumption, $[1, g'][g', g'w]$ is a geodesic, and then we get $d(g', \gamma) < C$ by the previous Claim. Thus, there exists $u \in \gamma$ such that $d(g', u) < 2C$ and
$d(1, u) = d(1, g')$. By (\ref{notstraight}), it follows that $d(u,
\gamma_+) \le d(1, w)$.

Setting $f := g'^{-1}u$, we have $f\in F_{g'}=F_g$ and so $d(1, gf) \le d(1, g)$. Denote by $k$ the element represented by the geodesic segment $[u, \gamma_+]_\gamma$.  Then $f k = ws$. Hence,   $$d(1, gws) =d(1, gf  k) \le d(1, gf) + d(u, \gamma_+) \le d(1, g) + d(1,
w).$$ This contradicts with (\ref{straight}). The first property holds: $g'ws \in
\Omega(g')$. The claim is proved.
\end{proof}

The finiteness of partial cone types now follows from the second
claim.
\end{proof}



\bibliographystyle{amsplain}
 \bibliography{bibliography}

\end{document}